\documentclass[12pt]{amsproc}

\newtheorem{theorem}{Theorem}[section]
\newtheorem{lemma}[theorem]{Lemma}

\theoremstyle{definition}

\theoremstyle{remark}

\numberwithin{equation}{section}
\usepackage{subfigure}
\usepackage{graphicx}
\usepackage{latexsym}
\usepackage{amsmath}
\usepackage{amssymb}
\usepackage{amsfonts}
\usepackage{verbatim}
\usepackage{mathrsfs}
\usepackage[latin1]{inputenc}
\usepackage{color}

\newtheorem{proposition}{Proposition}[section]

\begin{document}

\title[]
{Convergence analysis  of a symplectic  semi-discretization for stochastic NLS  equation with quadratic potential}

%    Remove any unused author tags.
%    author one information
\author{Jialin Hong}
\address{Institute of Computational Mathematics and Scientific/Engeering Computing, AMSS, Chinese Academy of Sciences,  Beijing, 100190, P.R. China\\
           School of Mathematical Science, University of Chinese Academy
of Sciences, Beijing, 100049, China}
\curraddr{}
\email{hjl@lsec.cc.ac.cn}
\thanks{}
%    author one information
\author{Lijun Miao}
\address{School of Mathematics, Liaoning Normal University, Dalian,  116029, P.R. China}
\curraddr{}
\email{miaolijun@lsec.cc.ac.cn (corresponding author )}
\thanks{}
%    author one information
\author{Liying Zhang}
\address{Department of Mathematics, College of Sciences, China University of Mining and Technology,  Beijing 100083, China. }
\curraddr{}
\email{lyzhang@lsec.cc.ac.cn}
\thanks{}
%    author two information

\subjclass[2010]{60H15, 60H35, 65P10}

\keywords{stochastic nonlinear Schr\"odinger equation, quadratic potential, additive noise, stochastic  symplectic scheme, stochastic multi-symplectic scheme}

\date{\today}

\dedicatory{}

\begin{abstract}
In this paper, we investigate the convergence in probability of a stochastic symplectic
scheme for stochastic nonlinear Schr\"{o}dinger equation with quadratic potential and an additive noise. Theoretical
analysis shows that our symplectic  semi-discretization is of order one in probability under appropriate regularity conditions for the initial value and noise.  Numerical
experiments are given to simulate the long time behavior of the discrete average charge and energy as well as  the influence of the external potential and noise, and to test the convergence order.\end{abstract}

\maketitle

%\tableofcontents

\section{Introduction}
\label{intro}

In this paper, we consider  the following stochastic nonlinear Schr\"{o}dinger equation with quadratic potential and additive noise
\begin{equation}\label{1.1}
\begin{split}
& \mathrm{i}d u+(\Delta u+\theta|x|^2 u+\lambda|u|^{2\sigma} u)dt=d W,
\quad (t,x)\in (0,T]\times  \mathcal{D},\\
&u(t,0)=u(t,1)=0,\;\;t\in (0,T],\\
&u(0,x)=u_0(x),\;\;x\in \mathcal{D},
\end{split}
\end{equation}
where $T\in(0,+\infty)$, $\mathcal{D}=(0,1)$, $\theta\in \mathbf{R}$,  $\lambda\neq 0$, $0<\sigma<2,$ and $W$ denotes a Wiener process expressing the random perturbations (\cite{KV1994}).  This equation models Bose--Einstein condensations %with $\theta<0$
under a 
magnetic trap when $\theta<0$, where the quadratic potential $|x|^2$ 
describes the magnetic filed whose role is to confine the movements of  particles (\cite{CR2002}). \cite{FZZ11} and \cite{MLT15} establish the well-posedness and blow up of the solution for \eqref{1.1}. The authors in \cite{MLT15}  indicate that the additive noise rather
than the potential dominates the dynamical behaviors of the solutions.  

It is known that numerical approximations have become an important tool to investigate
the behaviors of the solutions. In order to guarantee the reliability and effectiveness of numerical solutions for longtime simulations, we
expect numerical methods to preserve the intrinsic properties of the original systems
as much as possible. For Hamiltonian systems,
the symplectic schemes are shown to be superior to conventional ones especially in long time computation, artributed to their preservation of the qualitative property and the  symplectic  structure of the underlying continuous systems. 
The main goal of this work is to  analyze the convergence rate of the symplectic scheme for \eqref{1.1}.
We propose a mid-point method in  temporal direction of \eqref{1.1}  in order to preserve the properties of the original problems as much as possible and  to effectively simulate the influence of the external potential and noise on the long time behavior of the solution. It is shown that the mid-point semi-discretization is a symplectic scheme which preserves the  symplectic structure of  \eqref{1.1}.  The interested
readers are referred to \cite{BC2013} and references therein for the numerical simulation of the deterministic Schr\"odinger equation with potentials. We also refer to  \cite{CH16}  for the convergence analysis of the mid-point method  applied to the stochastic Schr\"odinger equation with Lipschitz coefficients, 
 to \cite{QC} for  the mean-square convergence of a symplectic local discontinuous Galerkin method to  stochastic linear Schr\"odinger equation with  a potential and multiplicative noise and to \cite{CA17} for the strong convergence rate of stochastic exponential method to  the stochastic linear Schr\"odinger equations with a multiplicative potential.

Furthermore,  we give the convergence order of the proposed scheme under non-Lipschitz condition. To this end, the higher regularity of the solution is needed due to the effect of semigroup. We get  the stability of solution of \eqref{1.1} in $H^4$ by means of the estimates for the high moments of charge and energy. Because the nonlinear term of \eqref{1.1} is not global Lipschitz, it is difficult to analyze the convergence order of the  symplectic scheme. Here we use  the truncated technique to get the truncated equation whose nonlinear term is global Lipschitz. Then we prove that the convergence order is one in probability for the symplectic scheme under appropriate hypothesis on initial value and noise. In addition, we simulate the long time behavior of the discrete average charge and energy under  the influence of the external potential and noise using a stochastic multi-symplectic scheme, owing to the multi-symplecticity of \eqref{1.1}. Here we cite \cite{HLMH06} and \cite{JWH13} as a partial list of the publications on
 the multi-symplectic scheme for the deterministic and stochastic Schr\"odinger equations. 
Numerical experiments present that the noise dominates the dynamics of the solution stronger than external potential.

The rest of the paper is organized as follows. Some properties of the solution, including the evolution law of charge, the uniform boundedness of energy and solution, are given in Section 2.  In Section 3, we first show that \eqref{1.1} owns the stochastic symplectic structure, then we construct
a stochastic symplectic scheme and  prove that its temporal order of
convergence is one in probability. In Section 4, we perform numerical experiments
to  test the convergence order in  Section 3, and to simulate the long time behavior of the discrete average charge and energy under the influence of the external potential and noise. In the remainder of the article, $C$ is a generic constant whose value may vary in different occurrences, $C(\cdot)$ denotes the constant depending on some parameters. 
\section{Stochastic NLS equation with quadratic potential }
\label{sec:1}
In order to state precisely Eq. \eqref{1.1},  we consider the probability space $(\Omega, \mathcal{F}, \mathbf{P})$ endowed with a normal filtration $\{\mathcal{F}_t\}_{t\geq 0}$. Let $\beta_k=\beta^1_k+\mathrm{i}\beta^2_k$ with $\beta^1_k$ and $\beta^2_k$  being its real and imaginary parts, respectively.  We assume that  $\{\beta^i_k\}_{k\in  \mathbf{N}, i=1,2}$ is a family of  real-valued independent indentified Brownian motions. Let $\{e_k\}_{k\in  \mathbf{N}}$ be an 
 orthonormal basis of some Hilbert space $U$.
 We consider the complex valued Wiener process
$$W=\sum_{k\in \mathbf{N}}\beta_k(t,\omega)\phi e_k(x),\quad t\in [0,T],\quad x\in \mathcal{D},\quad \omega \in \Omega,$$
where $\phi\in\mathcal{L}_2(U,H)$ the space of Hilbert--Schmidt operators from $U$ to  another Hilbert space $H$. The corresponding norm is then given by 
$$\| \phi \|^2_{\mathcal{L}_2 (U, H)}=tr(\phi^* \phi) =\sum_{k\in \mathbf{N}} \|\phi e_k \|^2 _{H} .$$
In addition, $ L^2(\mathcal{D})$ denotes Hilbert space with inner product $\Re\int_{\mathcal{D}}f(x)\overline{g}(x)dx$ for $f, g\in L^2(\mathcal{D})$. $\mathcal{L}_2 (U,H^s)$ is denoted by $\mathcal{L}^s_2$, where $H^s$ is Sobolev space consisting
of functions $f$ such that $\nabla^k f$ exist and are square integrable for all $k\in \{0,1,2\cdots s\}, s $ is positive integer. Throughout the paper, we assume that $\phi \in \mathcal{L}^s_2$ for a certain parameter $s$, and Sobolev space $\dot {H}^s = \{v\in  H^s;\;\; \Delta^j v=0 \;\; \text{on} \;\; \partial \mathcal{D},\;\;  \text {for} \;j\le s/2\}$  will be used . 
 
Now, we recall the mild solution of Eq. \eqref{1.1} from \cite{BD03}.

An $H_0^1$-valued $\{\mathcal{F}_t\}_{0\leq t\leq T} $-adapted process $\{u(t); t\in[0,T]\},$ is called a mild solution of \eqref{1.1} if for every $ t\in[0,T]$ holds $\mathbf{P}$-a.s.
 \begin{align}\label{1.2}
\nonumber u(t)=&S(t)u_0+\mathrm{i}\theta\int_0^t S(t-r) |x|^2 u(r)dr\\
&+\mathrm{i}\lambda\int_0^t S(t-r) |u(r)|^{2\sigma} u(r)dr-\mathrm{i}\int_0^t S(t-r)dW(r),
 \end{align}
where $S(t)=e^{\mathrm{i}t\Delta}$ denotes the semigroup of solution operator of the deterministic linear differential equation
$$ \mathrm{i}d u+\Delta udt=0~~\text{in}~~\mathcal{D},~~u=0~~\text{on}~~ \partial \mathcal{D} \times(0,T),~~u(0)=u_0~~\text{in}~~\mathcal{D}.$$

 If the noise term is eliminated in \eqref{1.1}, then it is the deterministic nonlinear Schr\"odinger equation with the quadratic potential
 \begin{align*}
\mathrm{i}d u+(\Delta u+\theta|x|^2 u+\lambda|u|^{2\sigma} u)dt=0,
\end{align*}
it possesses the charge conservation law
 \begin{align}\label{3.2}
 M(u(t))= \int_{ \mathcal{D}}|u(t,x)|^2dx=\int_{ \mathcal{D}}|u(t_0,x)|^2dx= M(u_0),
 \end{align}
and energy  conservation law
\begin{align}\label{3.390}
 H(u(t))=& \frac{1}{2}\int_{\mathcal{D}}|\nabla u(t,x)|^{2}dx -\frac{\theta}{2}\int_{\mathcal{D}}|xu(t,x)|^2dx-\frac{\lambda}{2\sigma+2}\int_{\mathcal{D}}|u(t,x)|^{2\sigma+2}dx\\
\nonumber =& H(u_0).
\end{align}

But these conservation laws of Eq. \eqref{1.1} are invalid.  Now, we state its charge evolution law and energy evolution laws, respectively.
 \begin{lemma}
   Eq. \eqref{1.1} has the following global charge evolution law a.s.
   \begin{align}\label{3.44}
 M(u(t))= M(u_0)-2\Im
 \sum_{k\in\mathbf{N}}\int_0^t\int_{\mathcal{D}}u(\overline{\phi e_k})dxd\beta_k(s)+t\|\phi\|^{2}_{\mathcal{L}^0_2},\quad\quad t\geq 0.
 \end{align}
 Moreover, for any $m\in  \mathbf{N},$ there exists a constant $C_m,$
\begin{align}\label{3.41}
\mathbf{E}\sup_{t\in[0,T]} M^m (u(t))\leq&C_m \big(\mathbf{E}(M^m (u_0))+1\big).
 \end{align}
\end{lemma}

\begin{proof}
The proof is based on the application of It\^o formula to functional $M(u)=\int_{ \mathcal{D}}|u(s,x)|^2dx.$ Since $M(u)$ is Fr\'echet derivable, the derivatives of  $M(u)$ along directions $\varphi$ and $(\varphi,\psi)$ are as follows:
$$DM(u)(\varphi)=2\Re\int_{ \mathcal{D}}u\overline{\varphi}dx, \quad D^2M(u)(\varphi,\psi)=2\Re\int_{\mathcal{D}}\varphi\overline{\psi}dx.$$
 From It\^o formula, we have
 \begin{align*}
 M(u)=& M(u_0)+\int_0^tDM(u(s))du+\frac{1}{2}\int_0^t tr[ D^2M(u)(-i\phi)(-i\phi)^*]ds\\
 =&M(u_0)+2\int_0^t\Re\int_{ \mathcal{D}}u(\overline{i\Delta u+i\theta|x|^2 u+i\lambda|u|^{2\sigma} u})dxds\\
 &+2\int_0^t\Re\int_{\mathcal{D}}u(-i)dxd\overline{W(s)}+\int_0^t tr(\phi\phi^*)ds\\
 =&M(u_0)-2\Im \int_0^t\int_{\mathcal{D}}udxd\overline{W(s)}+t\|\phi\|^{2}_{\mathcal{L}^0_2}\\
 =&M(u_0)-2\Im \sum_{k\in\mathbf{N}}\int_0^t\int_{\mathcal{D}}u\overline{\phi e_k}dxd\beta_k(s)+t\|\phi\|^{2}_{\mathcal{L}^0_2}.
  \end{align*}
 To prove \eqref{3.41}, we apply It\^o formula to $M^m (u),$
  \begin{align*}
\nonumber M^m (u)=& M^m (u_0)-2m\Im
 \sum_{k\in\mathbf{N}}\int_0^t M^{m-1} (u)\int_{\mathcal{D}}u \overline{\phi e_k}dxd\beta_k(s)+m\|\phi\|^{2}_{\mathcal{L}^0_2}\int_0^t M^{m-1} (u)ds\\
&+2m(m-1)\int_0^t M^{m-2} (u) \Re \sum_{k\in\mathbf{N}}\left( \int_{\mathcal{D}}u \overline{\phi e_k}dx \right)^2 ds,\quad\quad t\geq 0.
 \end{align*}
Taking the supremum and using a martingale inequality  yields
\begin{align*}
\mathbf{E}\sup_{t\in[0,T]} M^m (u(t))\leq& \mathbf{E}M^m (u_0)+6m\mathbf{E}\bigg[\bigg(\int_0^t M^{2m-2} (u)\|\phi^*u\|^2_{L^2}\bigg)^\frac{1}{2}\bigg]\\
&+2m(m-1)\mathbf{E}\bigg[\int_0^t M^{m-2} (u) \|\phi^*u\|^2_{L^2} ds\bigg]+m\|\phi \|^{2}_{\mathcal{L}^0_2}\mathbf{E}\bigg(\int_0^t M^{m-1} (u) ds\bigg)\\
\leq& \mathbf{E}M^m (u_0)+6mT\|\phi \|^{2}_{\mathcal{L}^0_2}\mathbf{E}\big[\sup_{t\in[0,T]} M^{m-\frac12} (u(t))   \big]\\
&+m(2m-1)T\|\phi \|^{2}_{\mathcal{L}^0_2}\mathbf{E}\big[\sup_{t\in[0,T]} M^{m-1} (u(t))\big].
\end{align*}
The lemma is proved  using H\"older and Young's inequalities in the second term of the right hand side and an induction  argument.
\end{proof}

Taking expectation in both sides of \eqref{3.44}, we have that 
 \begin{align}\label{3.45}
 \mathbf{E}M(u)= \mathbf{E}M(u_0)+t\|\phi\|^{2}_{\mathcal{L}^0_2},\quad\quad t\geq 0.
 \end{align}
The formula \eqref{3.45} indicates that the average charge is linear growth with respect to time $t.$ 
Next, we present the energy evolution law of  \eqref{1.1}, which can be proved by It\^o formula too.% the case of $x\in \mathbf{R} $  has established in  \cite{FZZ11}.

 \begin{lemma} 
 Eq. \eqref{1.1}
  has the following global energy evolution law a.s. 
 \begin{align}\label{3.5}
\nonumber H(u(t))=& H(u_0)+\Im\int_{ \mathcal{D}}\int_0^t(\Delta u+\theta|x|^2 u+\lambda|u|^{2\sigma} u)d\overline{W(s)}dx+\frac{1}{2}t \sum_{k\in\mathbf{N}}\|\nabla\phi e_k \|^{2}_{L^2}\\
\nonumber&- \frac{\theta}{2}t \sum_{k\in\mathbf{N}}\int_{ \mathcal{D}}|x|^2|\phi e_k |^{2}dx
 -\frac{\lambda}{2}\sum_{k\in\mathbf{N}}\int_0^t\int_{ \mathcal{D}}|u|^{2\sigma} |\phi e_k |^{2}dxds\\
 &-\sigma\lambda\sum_{k\in\mathbf{N}}\int_0^t\int_{\mathcal{D}}|u|^{2\sigma-2}(\Im(\overline{u}(\phi e_k)))^2dxds,\quad\quad t\geq 0.
 \end{align}
\end{lemma}

From the definition of energy in \eqref{3.390} and Gagliardo--Nirenberg's inequality, we can conclude the following result. 

\label{3.3}
\begin{lemma} \label{4.7}
Assume that $0<\sigma<2$. There exist constants $C(\theta)$ and $C(\sigma,\lambda)$ such that\\
(i) if $\lambda<0$, then $\|\nabla u\|_{L^2}^2\leq 2H(u)+C(\theta)\|u\|_{L^2}^2,$\\
(ii) if $\lambda>0$, then $ \|\nabla u\|_{L^2}^2\leq 4H(u)+C(\sigma,\lambda)\|u\|_{L^2}^{\frac{4+2\sigma}{2-\sigma}}+C(\theta)\|u\|_{L^2}^2.$
\end{lemma} 
Using this lemma, we get the uniform boundedness of $H(u(t))$.
 \begin{lemma}\label{2.3}
Let $p\geq1$, $\phi\in\mathcal{L}^1_{2} $, $u_0\in H_0^1$. There exists a constant $C\equiv C(p, T, \theta, \sigma, \lambda)$ such that
 \begin{align*}
&(i)~~\sup_{0\leq t\leq T} \mathbf{E}\big(H(u(t))\big)^p\leq C,\\
&(ii)~~\mathbf{E}\sup_{0\leq t\leq T}\big(H(u(t))\big)^p\leq C.
 \end{align*}
 \end{lemma}
 \begin{proof}
 We first consider the case of $p=1$. 
If $\lambda>0,$ applying the expectation to \eqref{3.5}, we have 
\begin{align*}
 \mathbf{E}H(u(t))\leq& \mathbf{E}H(u_0)+\frac{1}{2}t \mathbf{E}\sum_{k\in\mathbf{N}}|\nabla\phi e_k |^{2}_{L^2}+ \frac{|\theta|}{2}t \mathbf{E}\sum_{k\in\mathbf{N}}\int_{ \mathcal{D}}|x|^2|\phi e_k |^{2}dx\\
\leq&\mathbf{E}H(u_0)+tC\big(|\theta|\big)\|\phi\|^2_{\mathcal{L}^1_{2}},
 \end{align*}
the assertion (i) holds.
If $\lambda<0,$ 
\begin{align*}
\mathbf{E}H(u(t))\leq& \mathbf{E}H(u_0)+\frac{1}{2}t\mathbf{E} \sum_{k\in\mathbf{N}}\|\nabla\phi e_k \|^{2}_{L^2}+ \frac{|\theta|}{2}t \mathbf{E}\sum_{k\in\mathbf{N}}\int_{ \mathcal{D}}|x|^2|\phi e_k |^{2}dx\\
 &+|\lambda|\mathbf{E}\sum_{k\in\mathbf{N}}\int_0^t\int_{ \mathcal{D}}\bigg(\frac{1}{2}|u|^{2\sigma} |\phi e_k |^{2}+\sigma|u|^{2\sigma-2}\big(\Im(\overline{u}(\phi e_k))\big)^2\bigg)dxds\\
\leq& \mathbf{E}H(u_0)+tC\big(|\theta|\big)\|\phi\|^2_{\mathcal{L}^1_{2}}+|\lambda|\frac{1+2\sigma}{2}\mathbf{E}\sum_{k\in\mathbf{N}}\int_0^t\int_{ \mathcal{D}} |u|^{2\sigma}|\phi e_k |^{2}dxds.
 \end{align*}
Since H\"older inequality, Young's inequality and Gagliardo--Nirenberg's inequality,  we have
\begin{align*}
\sum_{k\in\mathbf{N}}\int_{ \mathcal{D}}|u|^{2\sigma}|\phi e_k |^{2}dx
&\leq \|u\|^{2\sigma}_{L^{2\sigma+2}}\|\phi \|^2_{\mathcal{L}_2(U,L^{2\sigma+2})}
\leq\frac{\sigma}{\sigma+1} \|u\|^{2\sigma+2}_{L^{2\sigma+2}}+\frac{1}{\sigma+1}\|\phi \|^{2\sigma+2}_{\mathcal{L}_2(U,L^{2\sigma+2})}\\
&\leq \frac{\sigma}{2}\|\nabla u\|^2_{L^{2}}+C(\sigma)\|u\|_{L^{2}}^{\frac{4+2\sigma}{2-\sigma}}+\frac{1}{\sigma+1}\|\phi \|^{2\sigma+2}_{\mathcal{L}_2(U,L^{2\sigma+2})}\\
&\leq \sigma H(u)+C(\theta)\|u\|_{L^2}^2+C\|u\|_{L^{2}}^{\frac{4+2\sigma}{2-\sigma}}+\frac{1}{\sigma+1}\|\phi \|^{2\sigma+2}_{\mathcal{L}^1_2},
 \end{align*}
 the last inequality follows from Lemma \ref{4.7} and $H^1$ is embedded into $L^{2\sigma+2}$. Therefore, 
$$ \mathbf{E}H(u(t))\leq Ct+C\mathbf{E}\int_0^t H(u(s))ds.$$
Then we use Gronwall's inequality to get the assertion (i).

If $p\geq 2,$ we apply It\^{o} formula to $\big(H(u)\big)^p,$ then
 \begin{align}\label{2.4}
\nonumber \big(H(u)\big)^p=&\big(H(u_0)\big)^p+\Im \int_\mathcal{D} \int_0^t p\big(H(u)\big)^{p-1}\big(\Delta u+\theta|x|^2 u+\lambda|u|^{2\sigma} u\big)d\overline{W}dx\\
\nonumber &+\frac{1}{2}\int_0^t p(p-1)\big(H(u)\big)^{p-2}\sum_{k\in\mathbf{N}}\left(\Im\int_\mathcal{D}\big(\Delta u+\theta|x|^2 u+\lambda|u|^{2\sigma} u\big) \overline{\phi e_k} dx\right)^2ds\\
\nonumber&+\frac{1}{2}\int_0^t p\big(H(u)\big)^{p-1} \bigg ( 
\sum_{k\in\mathbf{N}}\|\nabla\phi e_k \|^{2}_{L^2}- \theta\sum_{k\in\mathbf{N}}\int_{ \mathcal{D}}|x|^2|\phi e_k |^{2}dx\\
 & -\lambda\sum_{k\in\mathbf{N}}\int_{ \mathcal{D}}|u|^{2\sigma} |\phi e_k |^{2}dx
 -2\sigma\lambda\sum_{k\in\mathbf{N}}\int_{\mathcal{D}}|u|^{2\sigma-2}\big(\Im(\overline{u}(\phi e_k))\big)^2dx\bigg)ds.
 \end{align}
 Since the second term on the right-hand side vanishes after taking expectation, there remains to estimate the third term and the last term. For the third term, there exists a constant $C$ such that
\begin{align}\label{2.5}
&\sum_{k\in\mathbf{N}}\left( \Im\int_{\mathcal{D}}\big((\Delta u+\theta|x|^2 u+\lambda|u|^{2\sigma} u) \overline{\phi e_k} \big)dx\right)^2
\leq C\|\phi^*\big(\Delta u+\theta|x|^2 u+\lambda|u|^{2\sigma} u\big) \|^2_{L^2}.
 \end{align}
 The operator $\phi^*$ is bounded from $H^{-1}$ into $L^2$. Furthermore, $H^{1}$ is embedded into $L^{2\sigma+2}$, $\phi^*$ is also bounded from $L^{\frac{2\sigma+2}{2\sigma+1}}$ into $L^2$ and   Young's inequality, so we obtain
 \begin{align}\label{2.6}
\nonumber  &\|\phi^*\big(\Delta u+\theta|x|^2 u+\lambda|u|^{2\sigma} u\big) \|^2_{L^2}
\leq \|\phi\|^2_{\mathcal{L}^1_{2}}\big(\|\nabla u\|_{L^{2}}+|\theta||x|^2\|u\|_{L^2}+|\lambda|\|u\|^{2\sigma+1}_{L^{2\sigma+2}}\big)^2\\
&\nonumber \leq   3\|\phi\|^2_{\mathcal{L}^1_{2} }\big(\|\nabla u\|^2_{L^{2}}+\theta^2|x|^4\|u\|^2_{L^2}+\lambda^2\|u\|^{4\sigma+2}_{L^{2\sigma+2}}\big)\\
&\leq3\big(\|\phi\|^4_{\mathcal{L}^1_{2} }+\frac{\lambda^2}{2\sigma+2}\|\phi\|^{4\sigma+4}_{\mathcal{L}^1_{2} }+\frac{1}{2}\|\nabla u\|^4_{L^{2}}+\frac{\theta^4|x|^8}{2}\|u\|^4_{L^2}+\frac{\lambda^2(2\sigma+1)}{(2\sigma+2)}\|u\|^{4\sigma+4}_{L^{2\sigma+2}}\big).
 \end{align}
 From Gagliardo--Nirenberg's inequality and $0<\sigma<2,$
\begin{align}\label{2.7}
\|u\|^{4\sigma+4}_{L^{2\sigma+2}}\leq  C\|\nabla u\|^{2\sigma}_{L^{2}}\|u\|^{2\sigma+4}_{L^2}\leq C(\sigma)\big( \|\nabla u\|^{4}_{L^{2}}+\|u\|^{\frac{4\sigma+8}{2-\sigma}}_{L^2}\big).
 \end{align}
Substituting \eqref{2.7} into  \eqref{2.6}, and combining with \eqref{3.41} and Lemma \ref{4.7}, we have
\begin{align}\label{2.8}
\nonumber&\sum_{k\in\mathbf{N}}\left( \Im\int_{\mathcal{D}}\big(\Delta u+\theta|x|^2 u+\lambda|u|^{2\sigma} u\big) \overline{\phi e_k} dx\right)^2\\
\nonumber&\leq C(|\theta|, |\lambda|, \sigma)\bigg(\|\phi\|^4_{\mathcal{L}^1_{2}}+\|\phi\|^{4\sigma+4}_{\mathcal{L}^1_{2} }+\|\nabla u\|^4_{L^{2}}+\|u_0\|^4_{L^2}+\|u_0\|^{\frac{4\sigma+8}{2-\sigma}}_{L^2}\bigg)\\
&\leq C(|\theta|, |\lambda|, \sigma)\bigg(\big(H(u)\big)^2+ \|\phi\|^4_{\mathcal{L}^1_{2} }+\|\phi\|^{4\sigma+4}_{\mathcal{L}^1_{2} }+\|u_0\|^4_{L^2}+\|u_0\|^{\frac{4\sigma+8}{2-\sigma}}_{L^2}\bigg).
 \end{align}
For the last term in \eqref{2.4}, due to Young's inequality, Lemma \ref{4.7} and $H^{1}$ is embedded into $L^{2\sigma+2},$ we get\begin{align}\label{2.9}
\nonumber&\sum_{k\in\mathbf{N}}\|\nabla\phi e_k \|^{2}_{L^2}- \theta\sum_{k\in\mathbf{N}}\int_{ \mathcal{D}}|x|^2|\phi e_k |^{2}dx\\
\nonumber&-\lambda\sum_{k\in\mathbf{N}}\int_{ \mathcal{D}}|u|^{2\sigma} |\phi e_k |^{2}dx
 -2\sigma\lambda\sum_{k\in\mathbf{N}}\int_{\mathcal{D}}|u|^{2\sigma-2}\big(\Im(\overline{u}(\phi e_k))\big)^2dx\\
 \nonumber&\leq \|\phi\|^2_{\mathcal{L}^1_{2}}+C(|\theta|) \|\phi\|^2_{\mathcal{L}^1_{2}}+|\lambda|(2\sigma+1)\|u\|^{2\sigma}_{L^{2\sigma+2}}\|\phi\|^{2}_{L^{2\sigma+2}}\\
 &\leq C(|\theta|, |\lambda|, \sigma)\bigg(H(u)+\|\phi\|^2_{\mathcal{L}^1_{2}}+ \|\phi\|^{2\sigma+2}_{\mathcal{L}^1_{2}}+\|u_0\|^2_{L^2}+\|u_0\|^{\frac{2\sigma+4}{2-\sigma}}_{L^2}\bigg).
  \end{align}
  
Because of  \eqref{2.8}, \eqref{2.9}, and H\"older inequality, we deduce
\begin{align*}
 \mathbf{E}\big(H(u(t))\big)^p\leq Ct+C \mathbf{E}\int_0^{t}\big(H(u(s))\big)^pds,
\end{align*}
we apply Gronwall's inequality to obtain the estimate (i). 

To show the assertion (ii) for $p=1,$ we take the supremum over $t\in [0,T]$ in (3)
before taking the expectation. The main difference is
the appearance of the supremum of a stochastic integral compared to assertion (i).  This term can be estimated by a martingale inequality,
\begin{align*}
&\mathbf{E}\sup_{0\leq t\leq T}\bigg(\Im\int_{ \mathcal{D}}\int_0^t(\Delta u+\theta|x|^2 u+\lambda|u|^{2\sigma} u)d\overline{W(s)}dx)\bigg)\\
&\leq3\mathbf{E}
\bigg[\bigg(\int_0^T\|\phi^*(\Delta u+\theta|x|^2 u+\lambda|u|^{2\sigma} u)\|^2_{L^2} ds \bigg)^\frac{1}{2}\bigg]\\
&\leq C(\lambda, \theta, \|\phi\|_{\mathcal{L}^1_{2} })\mathbf{E}\bigg[\bigg(\int_0^T(\|\nabla u\|^2_{L^2}+ \|u\|^2_{L^2}+ \|u\|^{4\sigma+2}_{L^{2\sigma+2}} ) ds \bigg)^\frac{1}{2}\bigg]\\
&\leq C(T,\lambda, \theta, \|\phi\|_{\mathcal{L}^1_{2} })+\frac{1}{2}\mathbf{E}\big(\sup_{0\leq t\leq T}H(u(t)) \big).
\end{align*}
The assertion (ii) for $p \geq2$ uses arguments similar to the above estimate, so we skip the details here.
 \end{proof}
 
In \cite{BD03}, Theorem 4.6, a uniform boundedness for the energy is
used to construct a unique mild solution with continuous $H^1(\mathbf{R}^d)$-valued paths for stochastic nonlinear Schr\"odinger equation.
 We can follow the same strategy to construct the unique global mild
solution with continuous $H^1
(\mathcal{D})$-valued paths using Lemma \ref{2.3}.% in a bounded Lipschitz domain $\mathscr{D}\subset \mathbb{R}.$  
%Throughout the paper, we assume that  $0<\sigma<2$, so that the 
%equation \eqref{1.1} is well-posedness. 
Moreover,  we can get the same result in $H^4
(\mathcal{D})$  and have the following uniform boundedness of  the solution under $\sigma=1$. Similar to Theorem 2.1 of \cite{CHL16}, 
 The proof of the  stability of solution in Sobolev space $H^4$ is directly obtained by analyzing the functional 
\begin{align*}
f(u)=\|\nabla^{4}u\|^2_{L^2}-\lambda \Re\int_{\mathcal{D}}\big((-\Delta)^{3}u\big)(|u|^{2}\bar{u})dx.
\end{align*}
\begin{lemma}
Let $p\geq1$, $\sigma=1$, $u_0 \in  \dot{H}^4$ 
   and  $\phi \in\mathcal{ L}^4_2$.   There exists a constant $C=C(p, T,  u_0, \phi, \theta, \lambda, \sigma)$ such that 
\begin{align*}
\mathbf{E} \sup_{0\leq t\leq T}\|u \|^{2p}_{H^4}\leq C.
\end{align*}
\end{lemma}

\section{Stochastic symplectic scheme}
\label{sec:2}
As we all know, the stochastic Schr\"odinger equation without  quadratic potential is an infinite-dimensional stochastic Hamiltonian system, which characterize the geometric invariants of the phase flow and contribute to constructing the numerical schemes for long time computation.
In the following,  we show that Eq. \eqref{1.1} possesses stochastic symplectic structure.

Denote by $p$ and $q$ the real and imaginary parts of $u$, respectively. Let $\hat{W}=\sum\limits_{k\in \mathbf{N}}\beta_k(t,\omega)e_k(x)$ be a cylindrical Wiener process with $\hat{W}_1$ and $\hat{W}_2$ being its real and imaginary parts. Then the Wiener process $W=\phi \hat{W}_1+ \mathrm{i}\phi \hat{W}_2=:W_1+\mathrm{i}W_2$. Then Eq. \eqref{1.1} is equivalent to 
 \begin{align}\label{3.61}
\nonumber&d_t p+\Delta qdt+\theta|x|^{2} qdt+\lambda|p^2+q^2|^{\sigma} qdt=dW_2,\\
&d_t q-\Delta pdt-\theta|x|^{2} pdt-\lambda|p^2+q^2|^{\sigma} pdt=-d W_1
 \end{align}
with initial datum $(p(0),q(0))=(p_0,q_0).$
 Set 
$$H_1=- \frac{1}{2}\int_{\mathcal{D}}|\nabla u|^{2}dx +\frac{\theta}{2}\int_{ \mathcal{D}}|xu|^2dx+\frac{\lambda}{2\sigma+2}\int_{ \mathcal{D}}|u|^{2\sigma+2}dx.$$
Then the equation \eqref{3.61} can be rewritten as
 \begin{align}\label{4.1}
\nonumber&d_t p=-\frac{\delta H_1}{\delta q}dt+dW_2,\\
&d_t q=\frac{\delta H_1}{\delta p}dt-dW_1,
 \end{align}
where $\frac{\delta H_1}{\delta q} $ and $ \frac{\delta H_1}{\delta p}$ denote the variational derivative of $H_1$ with respect to $p$ and $q$, respectively.
In fact, \eqref{4.1} is an infinite-dimensional stochastic Hamiltonian system. Using the same procedure as \cite{CH16}, one can derive that \eqref{4.1} possesses the symplectic structure
 \begin{align}\label{4.2}
 \overline{\omega}(t):=\int_{x_0}^{x_1}dp\wedge dq dx.
  \end{align}

\begin{theorem}
The phase flow of Eq. \eqref{1.1} preserves the symplectic structure \eqref{4.2}.
\end{theorem}

In order to preserve the stochastic symplectic structure,
 we consider the mid-point  scheme of the temporal discretization for \eqref{1.1},
\begin{align}\label{4.4}
\mathrm{i}\frac{u^{n+1}-u^{n}}{\tau}+\Delta u^{n+\frac12}+\theta|x|^2 u^{n+\frac12}+\lambda|u^{n+\frac12}|^{2\sigma}u^{n+\frac12}=\frac{\triangle_n W}{\tau},\quad n=0, 1,\cdots,N-1,
\end{align}
where $\tau=\frac{T}{N}$ is time step size, $\triangle_n W=W(t_{n+1})-W(t_n)$,  $u^{n+\frac12}=\frac12(u^n+u^{n+1}).$  Denote
\begin{align*}
\overline{\omega}^{n}:=\int_{x_0}^{x_1}dp^{n} \wedge dq^{n} dx,
\end{align*}
we have the following result.
\begin{theorem}
The scheme \eqref{4.4} possesses the discrete symplectic structure, i.e.
\begin{align*}
\overline{\omega}^{n+1}=\overline{\omega}^{n}.
\end{align*}
\end{theorem}
\begin{proof}
For convenience, denote $\Phi(p, q)=\frac{1}{2\sigma+2}(p^2+q^2)^{\sigma+1},$ then  \eqref{4.4} can be rewritten as 
 \begin{align*}
\nonumber& p^{n+1}=p^{n}-\Delta q^{n+\frac12}\tau-\theta|x|^2 q^{n+\frac12} \tau -\lambda\frac{\partial \Phi}{\partial q}{(p^{n+\frac12}, q^{n+\frac12})} \tau+\triangle_n W_2,\\
&q^{n+1}=q^{n}+\Delta p^{n+\frac12}\tau+\theta|x|^2 p^{n+\frac12} \tau +\lambda\frac{\partial \Phi}{\partial p}{(p^{n+\frac12}, q^{n+\frac12})} \tau -\triangle_n W_1.
 \end{align*}
 Differentiating the above equation on the phase space, we obtain
  \begin{align*}
\nonumber& dp^{n+1}=dp^{n}-\Delta dq^{n+\frac12}\tau-\theta|x|^2 dq^{n+\frac12} \tau -\lambda\frac{\partial^2 \Phi}{\partial q\partial p}{dp^{n+\frac12}} \tau-\lambda\frac{\partial^2 \Phi}{\partial q^2}{dq^{n+\frac12}} \tau,\\
&dq^{n+1}=dq^{n}+\Delta dp^{n+\frac12}\tau+\theta|x|^2 dp^{n+\frac12} \tau +\lambda\frac{\partial^2 \Phi}{\partial p\partial q}{dq^{n+\frac12}} \tau+\lambda\frac{\partial^2 \Phi}{\partial p^2}{dp^{n+\frac12}} \tau.
 \end{align*}
Then
\begin{align*}
 &dp^{n+1}\wedge dq^{n+1}\\
 &=dp^{n}\wedge dq^{n}
 +\bigg(-\Delta dq^{n+\frac12}\tau-\theta|x|^2 dq^{n+\frac12} \tau -\lambda\frac{\partial^2 \Phi}{\partial q\partial p}{dp^{n+\frac12}} \tau-\lambda\frac{\partial^2 \Phi}{\partial q^2}{dq^{n+\frac12}} \tau\bigg)\wedge dq^{n}\\
&+dp^{n}\wedge\bigg(\Delta dp^{n+\frac12}\tau+\theta|x|^2 dp^{n+\frac12} \tau +\lambda\frac{\partial^2 \Phi}{\partial p\partial q}{dq^{n+\frac12}} \tau+\lambda\frac{\partial^2 \Phi}{\partial p^2}{dp^{n+\frac12}} \tau\bigg)\\
&+\bigg(-\Delta dq^{n+\frac12}\tau-\theta|x|^2 dq^{n+\frac12} \tau -\lambda\frac{\partial^2 \Phi}{\partial q\partial p}{dp^{n+\frac12}} \tau-\lambda\frac{\partial^2 \Phi}{\partial q^2}{dq^{n+\frac12}} \tau\bigg)\\
&\wedge\bigg(\Delta dp^{n+\frac12}\tau+\theta|x|^2 dp^{n+\frac12} \tau +\lambda\frac{\partial^2 \Phi}{\partial p\partial q}{dq^{n+\frac12}} \tau+\lambda\frac{\partial^2 \Phi}{\partial p^2}{dp^{n+\frac12}} \tau\bigg).
 \end{align*}
Substituting 
 \begin{align*}
\nonumber& dp^{n}=dp^{n+\frac12}+\Delta dq^{n+\frac12}\tau+\theta|x|^2 dq^{n+\frac12} \tau +\lambda\frac{\partial^2 \Phi}{\partial q\partial p}{dp^{n+\frac12}} \tau+\lambda\frac{\partial^2 \Phi}{\partial q^2}{dq^{n+\frac12}} \tau,\\
&dq^{n}=dq^{n+\frac12}-\Delta dp^{n+\frac12}\tau-\theta|x|^2 dp^{n+\frac12} \tau -\lambda\frac{\partial^2 \Phi}{\partial p\partial q}{dq^{n+\frac12}} \tau-\lambda\frac{\partial^2 \Phi}{\partial p^2}{dp^{n+\frac12}} \tau
 \end{align*}
into the above equality, we have
$$\int_{x_0}^{x_1}dp^{n+1}\wedge dq^{n+1}dx=\int_{x_0}^{x_1}dp^{n}\wedge dq^{n}dx,~~a.s.$$
\end{proof}

In fact,  we can prove that the scheme  \eqref{4.4} exists a numerical solution. 
\begin{proposition} 
Let $\phi \in \mathcal{L}^4_2,$ and $u_0$ be $\mathcal{F} _{0}$-measurable with values in $H_0^1$, then for sufficiently small $\tau$, there exists an $H_0^1(\mathcal{D})$-valued $\{\mathcal{F} _{t_n}\}_{0\leq n \leq N}$-adapted solution $\{u^n; 0\leq n \leq N\}$ of  \eqref{4.4} . 
\end{proposition}
\begin{proof}
Fix a family $\{\xi^n\}_{0\leq n\leq N-1}$ of deterministic functions in $H^1( \mathcal{D})$, we also fix $\widetilde{u}^n \in H_0^1(\mathcal{D}),$ the existence of solution  $\widetilde{u}^{n+1} \in H_0^1(\mathcal{D})$ of 
\begin{align}\label{4.6}
\mathrm{i}\frac{\widetilde{u}^{n+1}-\widetilde{u}^{n}}{\tau}+\Delta \widetilde{u}^{n+\frac12}+\theta|x|^2 \widetilde{u}^{n+\frac12}+\lambda|\widetilde{u}^{n+\frac12}|^{2\sigma}\widetilde{u}^{n+\frac12}=\xi^n
\end{align}
follows from a standard Galerkin method and Brouwer theorem (see %Lemma 3.1 in \cite{ADK91} or Lemma 3.1 in 
\cite{De2004A}).
Assuming that $\widetilde{u}^{n+1}\in H_0^1(\mathcal{D})$ is a solution of  \eqref{4.6}, multiplying \eqref{4.6} by the complex conjugate $\overline{\widetilde{u}}^{n+\frac12}$ of $\widetilde{u}^{n+\frac12}$, integrating over $\mathcal{D}$ and taking the imaginary part of the resulting identity, we have
\begin{align*}
\|\widetilde{u}^{n+1}\|_{L^2}^2 &=\|\widetilde{u}^{n}\|_{L^2}^2 +2\tau\Im \int_\mathcal{D } \xi^n {\overline{\widetilde{u}}}^{n+\frac12}dx\\
&\leq \|\widetilde{u}^{n}\|_{L^2}^2 +\tau\bigg(\|\xi^n\|_{L^2}^2+\frac12 \|\widetilde{u}^{n}\|_{L^2}^2 +\frac12 \|\widetilde{u}^{n+1}\|_{L^2}^2 \bigg).
\end{align*}
Therefore,
\begin{align}\label{4.8}
\|\widetilde{u}^{n+1}\|_{L^2}^2 \leq \frac{2+\tau}{2-\tau}\|\widetilde{u}^{n}\|_{L^2}^2 +\frac{2\tau}{2-\tau}\|\xi^n\|_{L^2}^2.
\end{align}
Using the same method,  we multiply \eqref{4.6} by $-\Delta \overline{\widetilde{u}}^{n+\frac12}-\theta|x|^2\overline{ \widetilde{u}}^{n+\frac12}-\lambda|\overline{\widetilde{u}}^{n+\frac12}|^{2\sigma}\overline{\widetilde{u}}^{n+\frac12},$ integrate over $\mathcal{D }$ and take the imaginary part of the resulting identity.  
Therefore, using H\"older inequality and Young's inequality, we obtain
\begin{align*}
&H(\widetilde{u}^{n+1})-H(\widetilde{u}^{n})\\
&=\tau \Im\int_\mathcal{D }\xi^n\bigg(-\Delta \overline{\widetilde{u}}^{n+\frac12}-\theta|x|^2\overline{ \widetilde{u}}^{n+\frac12}-\lambda|\overline{\widetilde{u}}^{n+\frac12}|^{2\sigma}\overline{\widetilde{u}}^{n+\frac12}\bigg)dx\\
&\leq \tau\bigg(\frac12\|\nabla\xi^n\|^2_{L^2}+ \frac14\|\nabla\widetilde u^n\|_{L^2}^2+\frac14\|\nabla\widetilde u^{n+1}\|_{L^2}^2+C(\theta)(\|\xi^n\|^2_{L^2}+ \|\widetilde{u}^{n}\|^2_{L^2}+ \|\widetilde{u}^{n+1}\|^2_{L^2}\big)\\
&~~~+\frac{|\lambda|}{2+2\sigma}\|\xi^n\|^{2+2\sigma}_{L^{2+2\sigma}}+\frac{|\lambda|(2\sigma+1)}{2+2\sigma}\|\overline{\widetilde{u}}^{n+\frac12}\|^{2+2\sigma}_{L^{2+2\sigma}}\bigg)\\
&\leq \tau\bigg(\frac12\|\nabla\xi^n\|^2_{L^2}+ \frac14\|\nabla\widetilde u^n\|_{L^2}^2+\frac14\|\nabla\widetilde u^{n+1}\|_{L^2}^2+C(\theta)(\|\xi^n\|^2_{L^2}+ \|\widetilde{u}^{n}\|^2_{L^2}+ \|\widetilde{u}^{n+1}\|^2_{L^2}\big)\\
&~~~+\frac{|\lambda|}{2+2\sigma}\|\xi^n\|^{2+2\sigma}_{L^{2+2\sigma}}+
\frac{\lambda|(2\sigma+1)}{4}\|\nabla\widetilde u^{n+\frac12}\|_{L^2}^2+C(\sigma, \lambda)\|\widetilde{u}^{n+\frac12}\|^{2+{\frac{4\sigma}{2-\sigma}}}_{L^2}\bigg)\\
&\leq \tau\bigg(\frac12\|\nabla\xi^n\|^2_{L^2}+ \frac14\|\nabla\widetilde u^n\|_{L^2}^2+\frac14\|\nabla\widetilde u^{n+1}\|_{L^2}^2+C(\theta)(\|\xi^n\|^2_{L^2}+ \|\widetilde{u}^{n}\|^2_{L^2}+ \|\widetilde{u}^{n+1}\|^2_{L^2})\\
&~~~+\frac{1}{2+2\sigma}\|\xi^n\|^{2+2\sigma}_{L^{2+2\sigma}}+
\frac{\lambda|(2\sigma+1)}{4}\|\nabla\widetilde u^{n}\|_{L^2}^2+\frac{\lambda|(2\sigma+1)}{4}\|\nabla\widetilde u^{n+1}\|_{L^2}^2\\
&~~~+C(\sigma, \lambda)\|\widetilde{u}^{n}\|^{2+{\frac{4\sigma}{2-\sigma}}}_{L^2}+C(\sigma, \lambda)\|\widetilde{u}^{n+1}\|^{2+{\frac{4\sigma}{2-\sigma}}}_{L^2}\bigg)\\
&\leq \frac{1}{4}\|\nabla\widetilde{u}^{n+1}\|_{L^2}^2+C\big(\tau, \sigma, |\lambda|, |\theta|, \|\xi^n\|_{H^1}, \|\widetilde{u}^{n}\|_{H^1},\|\widetilde{u}^{n+1}\|_{L^2}\big),
\end{align*}
where the second inequality follows from Gagliardo--Nirenberg's inequality and the last inequality follows from the fact that $\tau$ is sufficiently small. From \eqref{4.8} and Lemma \ref{4.7}, we have
\begin{align*}
\|\widetilde{u}^{n+1}\|^2_{H^1}\leq C\big(\tau, \sigma, |\lambda|, |\theta|,  \|\xi^n\|_{H^1}, \|\widetilde{u}^{n}\|_{H^1}\big).
\end{align*}

Define a map 
$$\Lambda: H_0^1\times H^1 \ni (\widetilde{u}^{n}, \xi^n)\to \Lambda(\widetilde{u}^{n}, \xi^n)\in \mathcal{P}(H_0^1),$$
where $\mathcal{P}(H_0^1)$ is the set of subsets of $H_0^1(\mathcal{D})$, $\Lambda(\widetilde{u}^{n}, \xi^n)$ is the set of solutions $\widetilde{u}^{n+1}$ of  \eqref{4.6}. From the closedness of the graph of $\Lambda$ and a selector theorem, there exists a universal and Borel measurable map $\kappa:H_0^1\times H^1 \to H_0^1$ such that $\kappa(u, \xi)\in \Lambda(u,\xi)$ for $(u,\xi)\in H_0^1\times H^1.$ Assume that $u^n \in H_0^1 $ is 
$\mathcal{F}_{t_n}$-measurable random variable, then $u^{n+1}=\kappa(u^n, \triangle_n W)$ is $H_0^1$-valued solution of \eqref{4.4}.
\end{proof}

Now we investigate the convergence  rate of the scheme \eqref{4.4} under $\sigma=1$. To
deal with the power law of the nonlinear term, we introduce a cut-off function
$\mu \in C^{\infty} (\mathcal{D})$ such that $\rm{supp} \mu \in [0,2]$ and $\mu=1$ on
 $[0,1]$ (see \cite{BDA06}).
  Write $\mu_{R}(u)=\mu\big(\frac{\|u\|_{H^1}}{R}\big)$, then we consider the truncated equation
 \begin{align}\label{4.9}
\mathrm{i}d u_R+\big(\Delta u_R+\theta|x|^2 u_R+\lambda\mu_{R}(u_R)|u_R|^{2} u_R\big)dt=d W.
\end{align}
Denote $f(u_R)=\mu_{R}(u_R)\big(|u_R|^{2} u_R\big)$,
the mild form of the corresponding mid-point  scheme of \eqref{4.9} is 
\begin{align}\label{4.10}
u_R^{n+1}=S_{\tau}u_R^{n}+\mathrm{i}\tau\theta |x|^2 T_\tau u_R^{n+\frac{1}{2}}+\mathrm{i}\tau\lambda T_\tau f(u_R^{n+\frac{1}{2}})-\mathrm{i}\sqrt{\tau}T_\tau \chi_{n+1},
   \end{align}
where
\begin{align*}
S_{\tau}=\bigg(1-\frac{ \mathrm{i}\tau}{2} \Delta\bigg)^{-1} \bigg(1+\frac{ \mathrm{i}\tau}{2} \Delta\bigg),~~T_{\tau}=\big(1-\frac{ \mathrm{i}\tau}{2} \Delta\big)^{-1}, ~~\chi_{n+1}=\frac{W(t_{n+1})-W(t_{n})}{\sqrt{\tau}}.
  \end{align*}
We have the following estimates to operators $S_{\tau}$ and $T_{\tau}$ ( \cite{BDA06}) with $\alpha\in[0,1]$, which are are useful in  the convergence analysis below :

$$\|S_{\tau}\|_{\mathcal{L}(L^2)}\leq1,~~~\|T_{\tau}\|_{\mathcal{L}(L^2)}\leq1,~~~,\|S_{\tau}-I\|_{\mathcal{L}(H^{2+\alpha}, H^\alpha)}\leq K\tau, $$
\begin{align}\label{4.999}
\|S(t_n)-S^{n}_{\tau}\|_{\mathcal{L}(H^{3+\alpha}, H^\alpha)}\leq K\tau~~~\|S(-s)-S^{-n}_{\tau}T_{\tau}\|_{\mathcal{L}(H^{3\alpha}, L^2)}\leq K\tau^\alpha.
  \end{align}

The scheme \eqref{4.10} is well-defined and $\{u_R^n\}_{0\leq n\leq N}$ 
are uniformly bounded provided
that the nonlinear term   $f(\cdot)$ is global Lipschitz. This Lipschitz continuity
can be guaranteed by 
Proposition 2.2 in \cite{BBD15} because $H^1$ is an algebra.

Assume that $u_0 \in  \dot{H}^4$ ,  $\phi\in\mathcal{ L}^4 _2$. From the global Lipschitz continuity of  the nonlinear term, we obtain that 
 \begin{align}\label{5.00}
\{f(u_R(t))\}_{t\in[0,T]}\in L^{2p}\big(\Omega;  L^{\infty}(0,T;\dot{H}^4)\big)
\end{align}
 and
$$\mathbf{E}\sup_{0\leq t\leq T}\|u_R(t)\|^{2p}_{\dot{H}^4}\leq C\big(p,T,R,\|u_0\|_{\dot{H}^4},\|\phi\|_{\mathcal{ L}^4 _2}\big)$$
using Sobolev embedding and Gronwall's inequality.
Now, we prove the following error estimate, the key of its proof lies in the mild solution and the unitarity of the both $S(t,r)$ and $S_{\tau}$.
\begin{proposition}\label{5.2}
Let $\sigma=1$, $u_0 \in  \dot{H}^4$ and $\phi\in \mathcal{L}^4 _2$, then for any $T\geq 0$, there exists a constant $C$ such that
 \begin{align*}
\mathbf{E}\max_{n=0,\cdots, N} \|u_{R}(t_n)-u_{R}^n\|^{2p}_{H^1}\leq C\tau^{2p}.
   \end{align*}
\end{proposition}
\begin{proof} For simplicity, we omit the dependence $R$ of $u^n$ and $u$.  %Due to $x\in\mathcal{ D}$, then there exists a constant $C({\theta})$ such that $C({\theta})=\sup\limits_{x\in \mathcal{ D}}\{|\theta| |x|^2\}$. 
Assume that $\tau<2/(C({\theta})+C(L_f, |\lambda|))$, here $L_f$ denotes the Lipschitz constant of $f$. Clearly,
 \begin{align*}
 u^n=&S^n_{\tau}u_0+\mathrm{i}\tau\theta\sum_{l=1}^{n} S^{n-l}_{\tau}T_{\tau} |x|^2 u^{l-\frac{1}{2}}+\mathrm{i}\tau\lambda \sum_{l=1}^{n} S^{n-l}_{\tau}T_{\tau} f(u^{l-\frac{1}{2}})-\mathrm{i}\sqrt{\tau} \sum_{l=1}^{n} S^{n-l}_{\tau}T_\tau \chi_l.
 \end{align*}
Define the mapping $g_n$ which maps $u^0,\cdots, u^{n-1}$ to $u^n$. For any sequences $\{u^n\}_{0\leq n\leq N}$  and $\{v^n\}_{0\leq n\leq N}$  with $u_0=v_0,$ we have
 \begin{align*}
&\|g_n(u^0,\cdots, u^{n-1})-g_n(v^0,\cdots, v^{n-1})\|_{H^1}\\
&\leq\tau\big\|\theta\sum_{l=1}^{n} S^{n-l}_{\tau}T_{\tau} |x|^2 (u^{l-\frac{1}{2}}- v^{l-\frac{1}{2}})\big\|_{H^1}+\tau\big\|\lambda \sum_{l=1}^{n} S^{n-l}_{\tau}T_{\tau} \big(f(u^{l-\frac{1}{2}})- f(v^{l-\frac{1}{2}})\big)\big\|_{H^1}\\
&\leq\tau C({\theta})\sum_{l=1}^{n-1}\|u^l-v^l\|_{H^1}+\frac{\tau C({\theta})}{2}\|g_n(u^0,\cdots, u^{n-1})-g_n(v^0,\cdots, v^{n-1})\|_{H^1}\\
&~~~+\tau C(L_f, |\lambda|)\sum_{l=1}^{n-1}\|u^l-v^l\|_{H^1}+\frac{\tau C(L_f, |\lambda|)}{2}\|g_n(u^0,\cdots, u^{n-1})-g_n(v^0,\cdots, v^{n-1})\|_{H^1}.
 \end{align*}
This fact provides that
 \begin{align}\label{4.11}
 \|g_n(u^0,\cdots, u^{n-1})-g_n(v^0,\cdots, v^{n-1})\|_{H^1}\leq  \frac{2\tau (C({\theta})+|C(L_f, |\lambda|))}{2-\tau (C({\theta})+C(L_f, |\lambda|))}\sum_{l=1}^{n-1}\|u^l-v^l\|_{H^1}.
  \end{align}
   We know that
 \begin{align*}
 u(t_n)=&S(t_n)u_0+\mathrm{i}\theta\int_0^{t_n} S(t_n-r) |x|^2 u(r)dr\\
&+\mathrm{i}\lambda\int_0^{t_n} S(t_n-r) f(u(r))dr-\mathrm{i}\int_0^{t_n} S(t_n-r)dW(r)
 \end{align*}
and deduce
\begin{align*}
& u(t_n)- g_n(u^0,\cdots, u(t_{n-1}))=\big(S(t_n)-S^n_{\tau}\big)u_0\\
& +\mathrm{i}\theta\sum_{l=1}^{n} \int_{t_{l-1}}^{t_l}\big(S(t_n-r) |x|^2 u(r)-S^{n-l}_{\tau}T_{\tau}  |x|^2 u(t_{l-\frac{1}{2}})\big)dr\\
&+ \mathrm{i}\lambda\sum_{l=1}^{n} \int_{t_{l-1}}^{t_l}\big(S(t_n-r) f(u(r))-S^{n-l}_{\tau}T_{\tau} f(u(t_{l-\frac{1}{2}}))\big) dr\\
 &-\mathrm{i} \sum_{l=1}^{n}\int_{t_{l-1}}^{t_l}\big( S(t_n-r)-S^{n-l}_{\tau}T_{\tau}\big)dW(r)\\
& =:A_n+B_n+C_n+D_n.
 \end{align*}
 From \eqref{4.999}, % we can obtain
  %$$\|S(-r)-S^{-l}_{\tau}T_{\tau}\|_{\mathcal{L}(H^3, L^2)}\leq C\tau$$
 %and
 %$$\|S(t_n)-S^{n}_{\tau}\|_{\mathcal{L}(H^3, L^2)}\leq C\tau.$$
 the terms $A_n$ and $D_n$ are estimated as following,
 \begin{align}\label{4.111}
 \mathbf{E}\max_{n=0,\cdots,N} \|A_n+D_n\|^{2p}_{H^1}\leq C\tau^{2p}\big(\mathbf{E}(\|u_0\|^{2p}_{H^4})+\|\phi\|^{2p}_{\mathcal{L}_2^4}\big).
 \end{align}
 For the term $B_n$, we divide it into the following parts
 \begin{align*}
&\mathrm{i}\theta\sum_{l=1}^{n} \int_{t_{l-1}}^{t_l}\big(S(t_n-r) |x|^2 u(r)-S^{n-l}_{\tau}T_{\tau}  |x|^2u(t_{l-\frac{1}{2}})\big)dr\\
&=\mathrm{i}\theta\sum_{l=1}^{n} \int_{t_{l-1}}^{t_l}\big(S(t_n-r)  -S^{n-l}_{\tau}T_{\tau}\big)|x|^2u(r)dr+\mathrm{i}\theta\sum_{l=1}^{n} \int_{t_{l-1}}^{t_l}S^{n-l}_{\tau}T_{\tau}|x|^2 \big(u(r)-u(t_{l-1})\big)dr\\
&~~~+\mathrm{i}\theta\sum_{l=1}^{n} \int_{t_{l-1}}^{t_l}S^{n-l}_{\tau}T_{\tau}|x|^2 \big(u(t_{l-1})-u(t_{l-\frac{1}{2}})\big)dr\\
&=:B^1_n+B^2_n+B^3_n.
 \end{align*}
 Concerning the first term $B^1_n$, we have
  \begin{align*}
& \bigg\|\mathrm{i}\theta\sum_{l=1}^{n} \int_{t_{l-1}}^{t_l}\big(S(t_n-r)  -S^{n-l}_{\tau}T_{\tau}\big)|x|^2u(r)dr\bigg\|_{H_1}\\
& \leq\bigg\|\mathrm{i}\theta\sum_{l=1}^{n} \int_{t_{l-1}}^{t_l}\big(S(t_n-r)-S(t_n)S^{-l}_{\tau}T_{\tau}\big)|x|^2u(r)dr\bigg\|_{H_1}\\
&+\bigg\|\mathrm{i}\theta\sum_{l=1}^{n} \int_{t_{l-1}}^{t_l}\big(S(t_n)-S^{n}_{\tau}\big)S^{-l}_{\tau}T_{\tau}|x|^2u(r)dr\bigg\|_{H_1}\\
&\leq\bigg\|\mathrm{i}\theta\sum_{l=1}^{n} \int_{t_{l-1}}^{t_l}\big(S(-r)-S^{-l}_{\tau}T_{\tau}\big)|x|^2u(r)dr\bigg\|_{H_1}\\
&+\|S(t_n)-S^{n}_{\tau}\|_{\mathcal{L}(H^3, L^2)}\bigg\|\mathrm{i}\theta\sum_{l=1}^{n} \int_{t_{l-1}}^{t_l}S^{-l}_{\tau}T_{\tau}|x|^2 u(r)dr\bigg\|_{H_4}\\
&\leq CT \tau \sup_{t\in [0,T]}\|u(t)\|_{H^4},
   \end{align*}
 %From Lemma 3.2 in \cite{BDA06}, we know that
%5$$\|S(-r)-S^{-l}_{\tau}T_{\tau}\|_{\mathcal{L}(H^3, L^2)}\leq C\tau$$
% and
% $$\|S(t_n)-S^{n}_{\tau}\|_{\mathcal{L}(H^3, L^2)}\leq C\tau,$$
%which conclude that
% \begin{align*}
% \|B^1_n\|_{H^1}\leq CT \tau \sup_{t\in [0,T]}\|u(t)\|_{H^4},
 % \end{align*}
 so that
   \begin{align*}
 \mathbf{E}\max_{n=0,\cdots,N}\|B^1_n\|^{2p}_{H^1} \leq C \tau^{2p} .
  \end{align*}
 In the context of mild solution of \eqref{4.9}, we have
  \begin{align*}
 u(r)=&S(r-t_{l-1})u(t_{l-1})+\mathrm{i}\theta\int_{t_{l-1}}^r S(r-s) |x|^2 u(s)ds\\
&+\mathrm{i}\lambda\int_{t_{l-1}}^r S(r-s)f(u(s))ds-\mathrm{i}\int_{t_{l-1}}^r S(r-s)dW(s).
  \end{align*}
Therefore,
  \begin{align*}
B^2_n
=&\mathrm{i}\theta\sum_{l=1}^{n} \int_{t_{l-1}}^{t_l} S^{n-l}_{\tau}T_{\tau}|x|^2\big(S(r-t_{l-1})-Id\big)u(t_{l-1})dr\\
&+\mathrm{i}\theta\sum_{l=1}^{n} \int_{t_{l-1}}^{t_l}S^{n-l}_{\tau}T_{\tau}|x|^2 \bigg(i\theta\int_{t_{l-1}}^r S(r-s) |x|^2 u(s)ds\bigg)dr\\
&+\mathrm{i}\theta\sum_{l=1}^{n} \int_{t_{l-1}}^{t_l}S^{n-l}_{\tau}T_{\tau}|x|^2 \bigg(i\lambda\int_{t_{l-1}}^r S(r-s)f(u(s))ds\bigg)dr\\
&-\mathrm{i}\theta\sum_{l=1}^{n} \int_{t_{l-1}}^{t_l} S^{n-l}_{\tau}T_{\tau}|x|^2\bigg(i\int_{t_{l-1}}^r S(r-s)dW(s)\bigg)dr\\
=:&B^{21}_n+B^{22}_n+B^{23}_n+B^{24}_n.
  \end{align*}
  For $\rho\in \mathbf{R},$ $S(\rho)$  is an isomery and
    \begin{align*}
 \|S(\rho)-Id\|_{\mathcal{L}(H^{4},H^2)} \leq C\rho.
   \end{align*}
  Thus,
   \begin{align*}
\|B^{21}_n\|_{H^1}\leq CT \tau\sup_{t\in[0,T]}\|u(t)\|_{H^4},
   \end{align*}
which leads to
    \begin{align*}
  \mathbf{E}\max_{n=0,\cdots, N}\|B^{21}_n\|^{2p}_{H^1}\leq C \tau^{2p}.
   \end{align*}
Because $\|B^{22}_n\|_{H^1}\leq CT\tau\sup\limits_{t\in[0,T]}\|u(t)\|_{H^4} $,  we have
  \begin{align*}
  \mathbf{E}\max_{n=0,\cdots, N}\|B^{22}_n\|^{2p}_{H^1}\leq C \tau^{2p}.
   \end{align*}
   Due to $\|B^{23}_n\|_{H^1}\leq CT\tau\sup\limits_{t\in[0,T]}\|f(u(t))\|_{H^4} $, we get%$\|B^{24}_n\|_{H^1}\leq CT\tau\|\phi\|_{H^1} $, 
  \begin{align*}
 \mathbf{E}\max_{n=0,\cdots, N}\|B^{23}_n\|^{2p}_{H^1}\leq C \tau^{2p}.
   \end{align*}
Using Fubini's theorem and martingale inequality, we have
 \begin{align*}
 &\mathbf{E}\max_{n=0,\cdots, N}\|B^{24}_n\|^{2p}_{H^1}\\
&= \mathbf{E}\max_{n=0,\cdots, N}\bigg\|-\mathrm{i}\theta\sum_{l=1}^{n} \int_{t_{l-1}}^{t_l} S^{n-l}_{\tau}T_{\tau}|x|^2\bigg(i\int^{t_{l}}_s S(r-s)dr\bigg)dW(s) \bigg\|^{2p}_{H^1} \\
& \leq C_p \mathbf{E}\bigg(\sum_{l=1}^{N} \int_{t_{l-1}}^{t_l}\sum_{k\in\mathbf{N}}\bigg\|\theta S^{n-l}_{\tau}T_{\tau}|x|^2\bigg(\int^{t_{l}}_s S(r-s)\phi e_kdr\bigg)\bigg \|^2_{\mathcal{L}_2^1}ds\bigg)^p
 \leq C \tau^{2p}.
   \end{align*}  
  
For $B^{3}_n$,  
     \begin{align*}
  \|B^{3}_n\|_{H^1}=&\bigg\|\mathrm{i}\theta\sum_{l=1}^{n} \int_{t_{l-1}}^{t_l}S^{n-l}_{\tau}T_{\tau} \big(|x|^2u(t_{l-1})-|x|^2u(t_{l-\frac{1}{2}})\big)dr \bigg\|_{H^1}  \\
\leq&CT\big\|u(t_{l-1})-u(t_{l-\frac{1}{2}})\big\|_{H^1}.  
  \end{align*} 
  Using the definition of $u(t)$, we can estimate  $B^{3}_n$ similarly.

It remains to estimate the term $C_n.$  It can be decomposed into a sum
  \begin{align*}
  &\mathrm{i}\lambda\sum_{l=1}^{n} \int_{t_{l-1}}^{t_l}\big(S(t_n-r) f(u(r))-S^{n-l}_{\tau}T_{\tau} f(u(t_{l-\frac{1}{2}}))\big)dr \\
 & = \mathrm{i}\lambda\sum_{l=1}^{n} \int_{t_{l-1}}^{t_l}\big(S(t_n-r)-S^{n-l}_{\tau}T_{\tau}\big) f(u(r))dr\\
&~~~+ \mathrm{i}\lambda\sum_{l=1}^{n} \int_{t_{l-1}}^{t_l}S^{n-l}_{\tau}T_{\tau} \big(f(u(r))-f(u(t_{l-1}))\big)dr \\
&+\mathrm{i}\lambda\sum_{l=1}^{n} \int_{t_{l-1}}^{t_l}S^{n-l}_{\tau}T_{\tau} \big(f(u(t_{l-1}))-f(u(t_{l-\frac{1}{2}}))\big)dr\\
 & =: C^1_n+C^2_n+C^3_n.
  \end{align*}
For the term $C^1_n$, we have
  \begin{align*}
& \bigg\|\mathrm{i}\lambda\sum_{l=1}^{n} \int_{t_{l-1}}^{t_l}\big(S(t_n-r)-S^{n-l}_{\tau}T_{\tau}\big) f\big(u(r)\big)dr\bigg\|_{H_1}\\
& \leq\bigg\|\mathrm{i}\lambda\sum_{l=1}^{n} \int_{t_{l-1}}^{t_l}\big(S(t_n-r)-S(t_n)S^{-l}_{\tau}T_{\tau}\big)f\big(u(r)\big)dr\bigg\|_{H_1}\\
&+\bigg\|\mathrm{i}\lambda\sum_{l=1}^{n} \int_{t_{l-1}}^{t_l}\big(S(t_n)-S^{n}_{\tau}\big)S^{-l}_{\tau}T_{\tau}f\big(u(r)\big)dr\bigg\|_{H_1}\\
&\leq\bigg\|\mathrm{i}\lambda\sum_{l=1}^{n} \int_{t_{l-1}}^{t_l}\big(S(-r)-S^{-l}_{\tau}T_{\tau}\big)f\big(u(r)\big)dr\bigg\|_{H_1}\\
&+\|S(t_n)-S^{n}_{\tau}\|_{\mathcal{L}(H^3, L^2)}\bigg\|\mathrm{i}\lambda\sum_{l=1}^{n} \int_{t_{l-1}}^{t_l}S^{-l}_{\tau}T_{\tau} f\big(u(r)\big)dr\bigg\|_{H_4}.
   \end{align*}
Similar to the estimates of $B^1_n$, we have
    \begin{align*}
 \|C^1_n\|_{H^1}\leq C T\tau\sup_{t\in[0,T]}\|f(u(t))\|_{H^4}.
   \end{align*}
   Therefore, 
   \begin{align*}
 \mathbf{E}\max_{n=0,\cdots,N}\|C^1_n\|^{2p}_{H^1} \leq C \tau^{2p} .
  \end{align*}
 Using the definition of the mild solution and Taylor formula,
  \begin{align*}
 C^2_n=& \mathrm{i}\lambda\sum_{l=1}^{n} \int_{t_{l-1}}^{t_l}S^{n-l}_{\tau}T_{\tau} f^{\prime}(u(t_{l-1}))\bigg((S(r-t_{l-1})-Id)u(t_{l-1})\\
 &+\mathrm{i}\theta\int_{t_{l-1}}^r S(r-s) |x|^2 u(s)ds+i\lambda\int_{t_{l-1}}^r S(r-s)f(u(s))ds\bigg)dr\\
 & -\sum_{l=1}^{n} \int_{t_{l-1}}^{t_l}S^{n-l}_{\tau}T_{\tau} f^{\prime}(u(t_{l-1}))
 \left( \int_{t_{l-1}}^r S(r-s)dW(s)\right)dr\\
 &+\frac{\mathrm{i}}{2}\sum_{l=1}^{n} \int_{t_{l-1}}^{t_l}S^{n-l}_{\tau}T_{\tau} \int_0^1 f^{{\prime}{\prime}}(\rho u(r)+(1-\rho)u(t_{l-1}))\\
 &\cdot( u(r)-u(t_{l-1}), u(r)-u(t_{l-1}))d\rho dr.
 \end{align*}
 Combined with the estimate of $B_n$ and \eqref{5.00}, 
   \begin{align*}
   \mathbf{E}\max_{n=0,\cdots, N}\|C^2_n\|^{2p}_{H^1} \leq C \tau^{2p} .
   \end{align*}
 For the term $C^3_n$, the estimate of $B_n^3$ and Lipschitz continuity of $f$ yield that
    \begin{align*}
\mathbf{E} \max_{n=0,\cdots, N}\|C^3_{n}\|^{2p}_{H^1}&\leq  C\tau^{2p}.
 \end{align*}
Summarize the above estimates, we obtain
  \begin{align*}
\mathbf{E} \max_{n=0,\cdots, N}\|u(t_{n} )-g_{n}(u^0,\cdots u(t_{n-1} ))\|^{2p}_{H^1}\leq C\tau^{2p}.
 \end{align*}
Then, by \eqref{4.11} and Minkowski inequality, we have
\begin{align*}
&\big( \mathbf{E} \max_{\tilde{n}=0,\cdots, n}\| u^{\tilde{n}}-u(t_{\tilde{n}})\|^{2p}_{H^1}\big)^{1/2p}\\
&\leq C\frac{2(C({\theta})+C(|\lambda|, L_f))}{2-\tau (C({\theta})+C(|\lambda|, L_f))}\tau\sum_{l=1}^{n-1}\bigg(\mathbf{E}\big( \max_{\tilde{l}=0,\cdots, l}\| u^{\tilde{l}}-u(t_{\tilde{l}})\|^{2p}\big)\bigg)^{1/2p}+C\tau.
 \end{align*}
 The result follows from the discrete Gronwall's inequality.
\end{proof}

\begin{theorem}
 Let $\sigma=1$, $u_0 \in  \dot{H}^4$ and $\phi\in \mathcal{L}^4 _2$, then for any $0\leq n\leq N$, 
  $$\lim_{C\to\infty}\mathbf{ P}\big(\max_{n=0,\cdots, N}\|u(t_n)-u^n\|_{H^1}\geq C\tau\big)=0.$$
\end{theorem}
\begin{proof}
Define a stopping time 
   \begin{align*}
   \iota_R=\inf_{n=0,\cdots, N} {\{t_n:\|u_R^{n-1}\|_{H^1}\geq R ~~or~~ \|u_R^{n}\|_{H^1}\geq R \}},
      \end{align*}
and the discrete solution $u^n=u_R^n$ if $t_n \leq \iota_R.$ From Proposition \ref{5.2}, for any $R>0$, 
\begin{align*}
\mathbf{E}\max_{n=0,\cdots, N} \|u_R(t_n)-u_R^n\|^{2p}_{H^1}\leq C(R,p)\tau^{2p}.
   \end{align*}
This yields that $\max\limits_{n=0,\cdots, N} \|u_R(t_n)-u_R^n\|_{H^1}$ converges to 0 in probability as $\tau \to 0$. Similar to \cite{BDA06}, by
Chebyshev inequality, we have
\begin{align*}
&\mathbf{P}\big(\max_{n=0,\cdots, N}\|u(t_n)-u^n\|_{H^1}\geq C\tau\big)\\
\leq &\mathbf{P}\big(\max_{n=0,\cdots, N} \|u(t_n)\|_{H^1}\geq R\big)+\mathbf{P}\big(\max_{n=0,\cdots, N} \|u^n\|_{H^1}\geq R\big)\\
&+\mathbf{P}\big(\max_{n=0,\cdots, N}\|u_{R}(t_n)-u_{R}^n\|_{H^1}\geq C\tau\big)\\
\leq &\mathbf{P}\big(\max_{n=0,\cdots, N} \|u(t_n)\|_{H^1}\geq R\big)+\mathbf{P}\big(\max_{n=0,\cdots, N} \|u^n\|_{H^1}\geq R\big)\\
&+\frac{\mathbf{E}\max\limits_{n=0,\cdots, N}\|u_{R}(t_n)-u_{R}^n\|^{2}_{H^1}}{C^2\tau^2}.
\end{align*}
Following from the uniform boundedness of $u^n$ and $u(t_n)$ together with Proposition \ref{5.2}, we have 
$\mathbf{P}\big(\max_{n=0,\cdots, N}\|u(t_n)-u^n\|_{H^1}\geq C\tau\big)$ converges to 0 as $C$ and $R$ tend to $\infty.$ 
\end{proof}
\section{Numerical experiments}

In this section, we focus on the following example
\begin{equation}\label{5.1}
  \begin{split}
& \mathrm{i}d u(t,x)+(\Delta u(t,x)+ \theta x^2 u(t,x)+\lambda|u(t,x)|^{2} u(t,x))dt=\epsilon  d W,\\
    &u(0,x)=\sin(\pi x).
  \end{split}
\end{equation}
Here, $\epsilon$ denotes the size of noise and $\epsilon=0$ can be considered as the deterministic case in some sense.
Next, we first present numerical experiments to verify the convergence order of the proposed stochastic symplectic scheme \eqref{4.4} on $[0,1]\times[0,T]$.
In order to  investigate the influence of quadratic potential  and noise, we  give some  numerical experiment on the solution and evolution laws of charge and energy in the sense of expectation for stochastic multi-symplectic scheme\begin{align}\label{6.10}
   \mathrm{i} (\delta^+_t u^n_{j+\frac{1}{2}}+\delta^+_t u^n_{j-\frac{1}{2}})
     =\nonumber&-2\delta^+_x\delta^-_x u^{n+\frac{1}{2}}_{j}
     -\theta|x_{j+\frac12}|^2 u^{n+\frac{1}{2}}_{j+\frac{1}{2}}
    -\theta|x_{j-\frac12}|^2 u^{n+\frac{1}{2}}_{j-\frac{1}{2}}\\
    -\lambda|u^{n+\frac{1}{2}}_{j+\frac{1}{2}}|^{2\sigma} u^{n+\frac{1}{2}}_{j+\frac{1}{2}}
    & -\lambda|u^{n+\frac{1}{2}}_{j-\frac{1}{2}}|^{2\sigma} u^{n+\frac{1}{2}}_{j-\frac{1}{2}}
 +\frac{\epsilon}{\tau}(\triangle_n W_{j+\frac12}+\triangle_n W_{j-\frac12}).
\end{align}
where
$$\delta^+_x u_j:= \frac{u_{j+1}-u_{j}}{h} ,~\delta^+_t u_n:= \frac{u^{n+1}-u_{n}}{\tau},~\delta^+_x\delta^-_x u_j:= \frac{u_{j+1}-2u_{j}+u_{j-1}}{h^2}.$$
It is obtained by applying mid-point scheme to \eqref{5.1} in both temporal and spatial directions \cite{JWH13}.

Under the homogeneous Dirichlet boundary condition, \eqref{6.10} possesses the 
discrete charge and energy properties deduced by similar  method to \cite{JWH13}, respectively:
\begin{align}\label{3.11}
 h \sum_j |u^{n+1}_{j+\frac{1}{2}}|^2 =h\sum_j |u^{n}_{j+\frac{1}{2}}|^2+h\epsilon\Im  \sum_j \bigg( {\triangle_n W_{j+\frac12}}+{\triangle_n W_{j-\frac12}}\bigg)\overline{u^{n+\frac{1}{2}}_{j}},
\end{align}
and
\begin{align}\label{3.12}
\nonumber  
H^{n+1}-H^{n}=&\frac{\lambda }{2}h\sum_j |u^{n+\frac{1}{2}}_{j+\frac{1}{2}}|^{2}\big(|u^{n+1}_{j+\frac{1}{2}}|^2-|u^{n}_{j+\frac{1}{2}}|^2\big)-\frac{\lambda }{2\sigma+2}h\sum_j\big(|u^{n+1}_{j+\frac{1}{2}}|^{4}-|u^{n}_{j+\frac{1}{2}}|^{4}\big)\\
&%\nonumber-\theta h\sum_jx_{j+\frac12}^2\big(|u^{n+1}_{j+\frac{1}{2}}|^{2}-|u^{n}_{j+\frac{1}{2}}|^{2}\big)
+ h\epsilon \Re \sum_j \big(\triangle_n W_{j+\frac12}+\triangle_n W_{j-\frac12}\big)(\delta^+_t\overline{u^{n}_{j}}),\end{align}
Here, the global energy of \eqref{6.10} at time $t_n$ is defined as
$$H^{n}=\frac{1}{2}h\sum_j|\delta_x^+u_j^n|^2-\frac{\theta}{2} h\sum_j|x_{j+\frac{1}{2}}u^{n}_{j+\frac{1}{2}}|^{2}-\frac{\lambda}{2\sigma+2}h\sum_j |u^{n}_{j+\frac{1}{2}}|^{2\sigma+2}.$$

In the sequel, we choose $\theta=1, \lambda=1$, and consider 
the real-valued  Wiener process $W(t,x) =\sum^{M}_{k=1} \frac{1}{k^{4.6}}e_k(x)\beta_{k} (t)$ with the truncated number $M=50$, where
 $\{\beta_{k}; 1 \leq k \leq M\}$ is a family of independent $\mathbf{R}$--valued
Brownian motions,  and $e_k(x)= \sin(\pi kx),k=1,... ,M$ denotes the orthonormal basis of $L^2(0,1)$.
Let $I_{\tau}=\{t_n; 0 \leq n \leq N\}$ be the uniform discretization of $[0, 1]$ of size $\tau > 0$, 
 and apply the uniform discretization of $[0, 1]$ of size $h = \frac{1}{256}$. The reference values are generated for the smallest mesh size $ \tau_{ref}= 2^{-14}.$  In Fig.4.1, we plot the convergence curves based on the errors $\|u_{ref}-u_{\tau}\|_{L^2}$ with  $ \tau= 2^p \tau_{ref}, p=1, 2, 3, 4, 5.$ We can see that  the convergence order for the $L^2$-error of the mid-point  scheme is 2 if $\epsilon=0$ and the slope of
our scheme \eqref{4.4} in stochastic case is 1. This observation verifies the theoretical result in Section 3.
\begin{figure*}
\includegraphics[width=6cm,height=5cm]{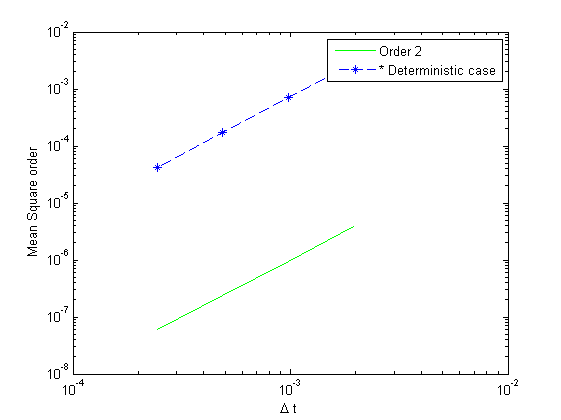}\hfill
\includegraphics[width=6cm,height=5cm]{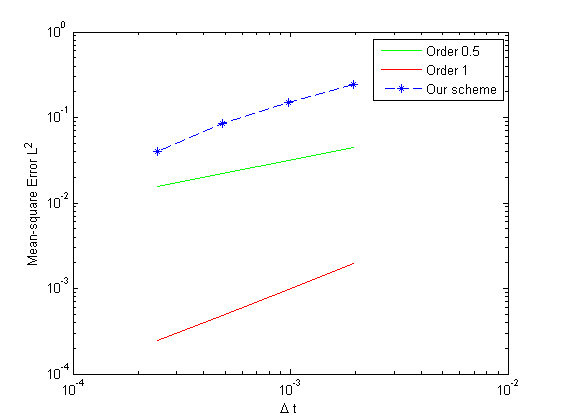}
\caption{Rates of convergence for $\epsilon=0$ (left) and $\epsilon=\sqrt{2}$ (right). }
\label{fig:2}       % Give a unique label
\end{figure*}

We now investigate the behaviors of solitary wave solution under the influence of quadratic potential and noise.
In these experiments, we take $\lambda=1, \sigma=1$ and the step sizes $h = 0.1$, $ \tau=0.01$. 
 The profiles of amplitude $|u(t,x)|$ are presented in Fig.4.2 and Fig.4.3. 
In Fig.4.2, these two figures give the propagation of solitary wave when taking different size of noise $\epsilon=0.02,~0.2$ and $\theta=-1$.
We find that the waveform of solution is obviously disturbed as the scale of noise becomes larger, that is the velocity of solitary wave is influenced.
%In fact, the noise amplitude is higher and indeed influences the velocity of solitary wave.
Fig.4.3 presents the long time behaviors of solution when we take the different kind of quadratic potential $\theta=-1,0,1$ with $\epsilon=0.2$.
Combining these three figures, we find that the external potential influences the
 velocity of solitary wave, it can neither prevent the propagation, nor destroy the solitary.  Moreover, it dominates the dynamics of
the solution weaker than noise.

\begin{figure*}
 \includegraphics[width=6cm,height=5cm]
 {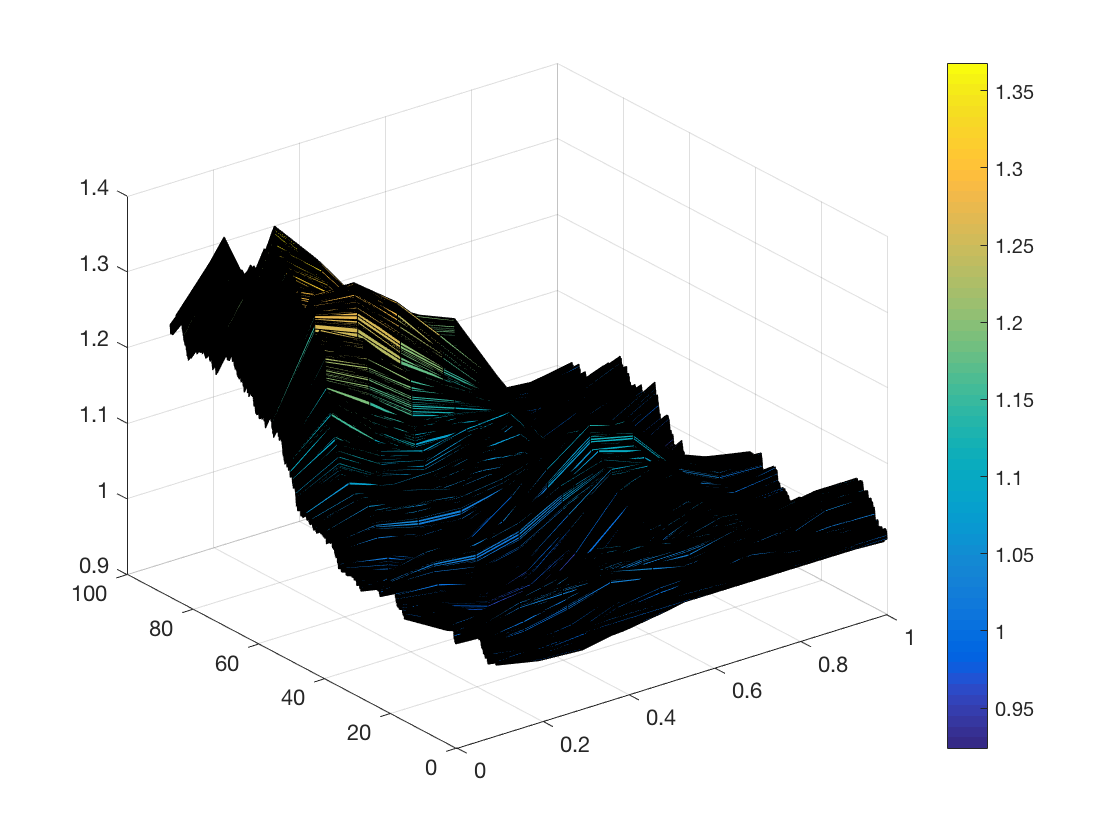}
 \includegraphics[width=6cm,height=5cm]
 {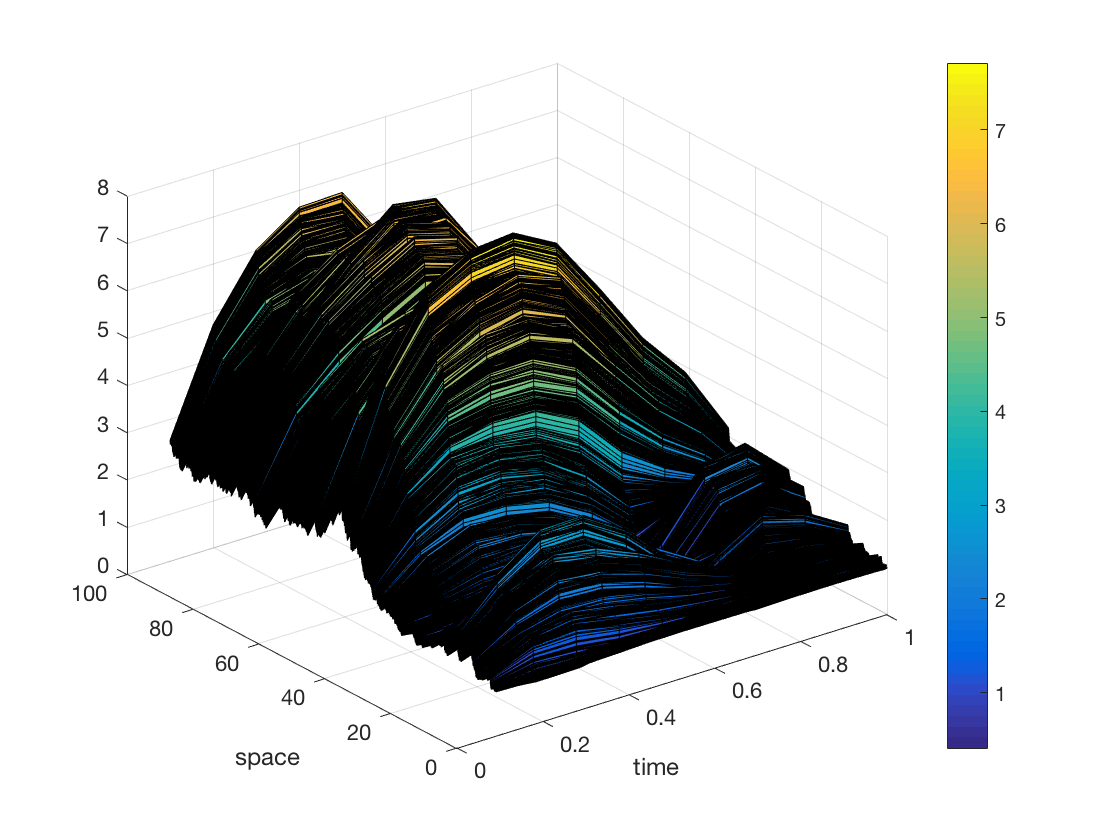}
\caption{The profile of numerical solution $|u(x,t)|$ for one trajectory with different noise when $\theta=-1$. The left figure is the case of $\epsilon=0.02$, The right figure is the case of $\epsilon=0.2.$}
\label{fig:2}       % Give a unique label
\end{figure*}

\begin{figure}[h]
   \includegraphics[width=3.8cm,height=3.8cm]{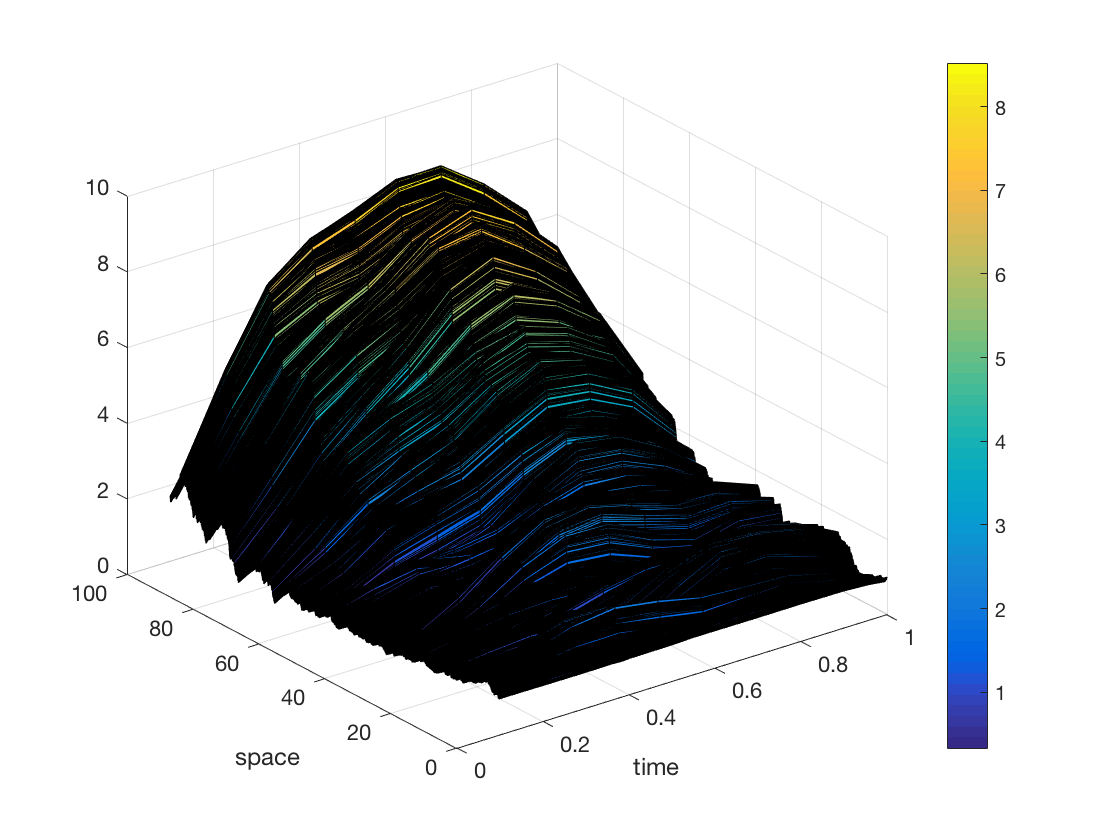}
\includegraphics[width=3.8cm,height=3.8cm]{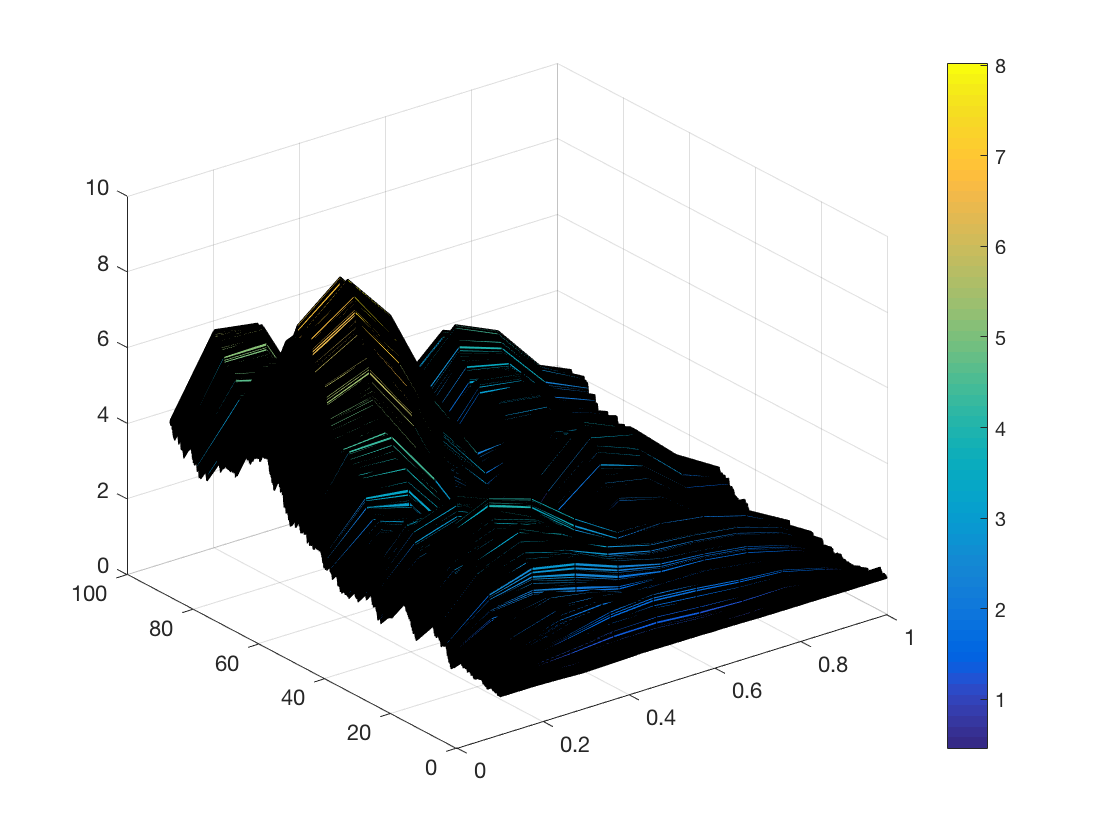}
\includegraphics[width=3.8cm,height=3.8cm]{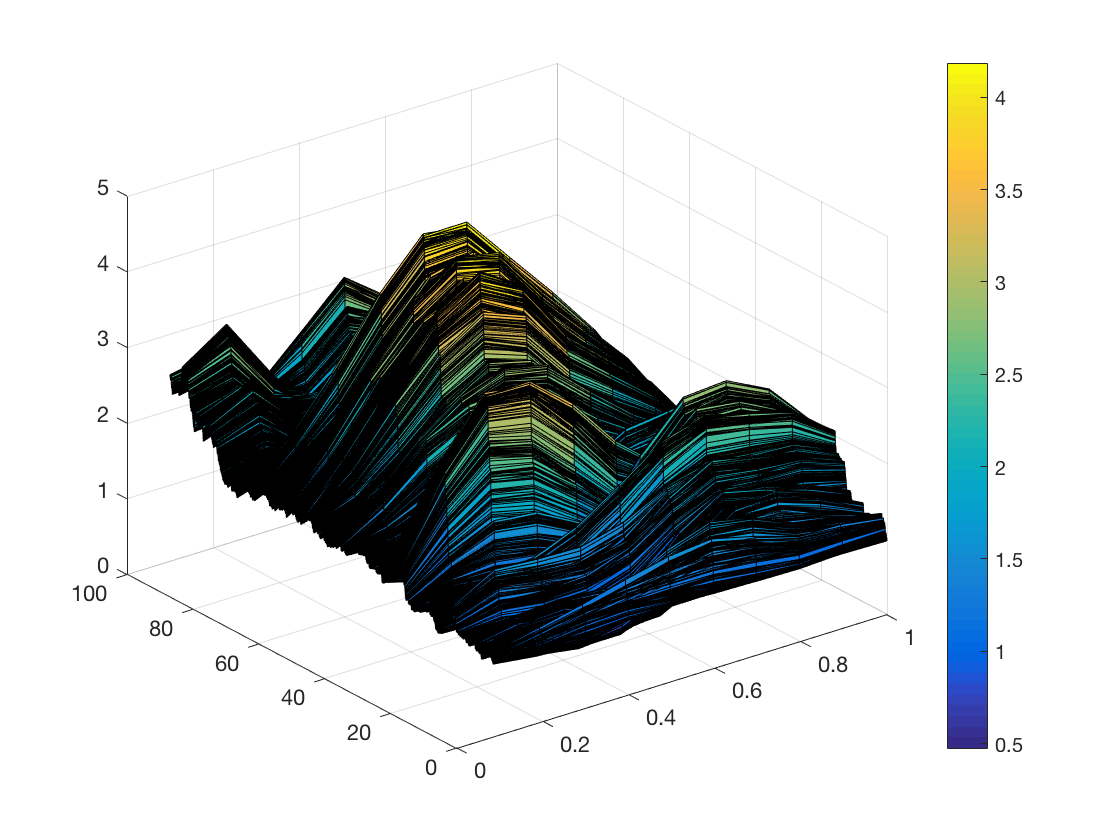}
\caption{The profile of $|u(t,x)|$ when $\theta=-1$ (left),~$0$ (middle),~$1$ (right), respectively.}
\end{figure}

The average charge conservation law $\mathbf{E}(M(u(t,x)))$ follows linearly  grow evolution law with respect to time and the average energy conservation law $\mathbf{E}(H(u(t,x)))$ follows linear evolution law at most. These phenomena are reflected in Fig.4.4 and Fig4.5, respectively, where the evolution of the average discrete charge and energy obey nearly linear growth over 100 trajectories.  In Fig.4.4, the different external potential  have  small effects on the average charge and energy. But different size noises have obvious  effects on them,  the evolution laws of the charge and energy more and more tend to the conservation laws when $\epsilon$ tends to $0$ especially.
\begin{figure}[h]
 \includegraphics[width=6cm,height=5cm]{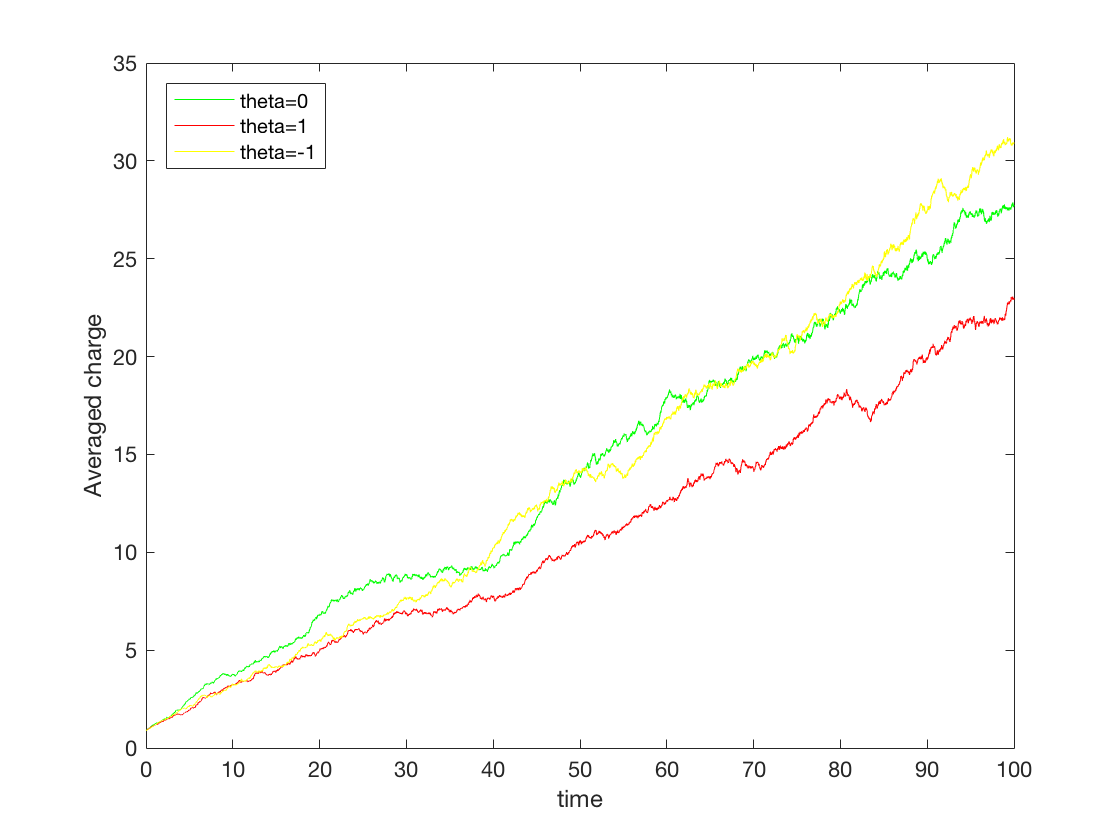}
\includegraphics[width=6cm,height=5cm]
{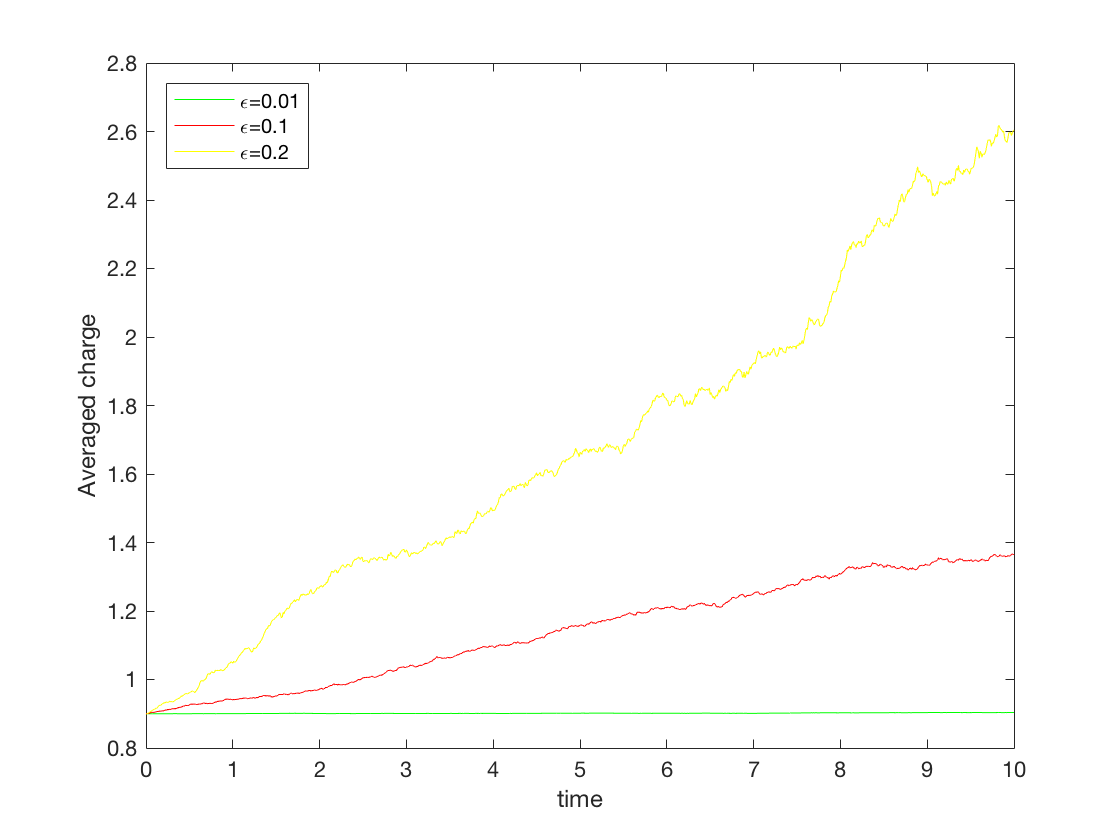}
\caption{The evolution of the average discrete charge obeys linear growth if $\theta=-1, 0, 1, \lambda=1, \sigma=1$ (left);  $\epsilon=0.01, 0.1, 0.2,  \theta=-1, \lambda=1, \sigma=1$ (right).}
\end{figure}

\begin{figure}[h]
\includegraphics[width=6cm,height=5cm]{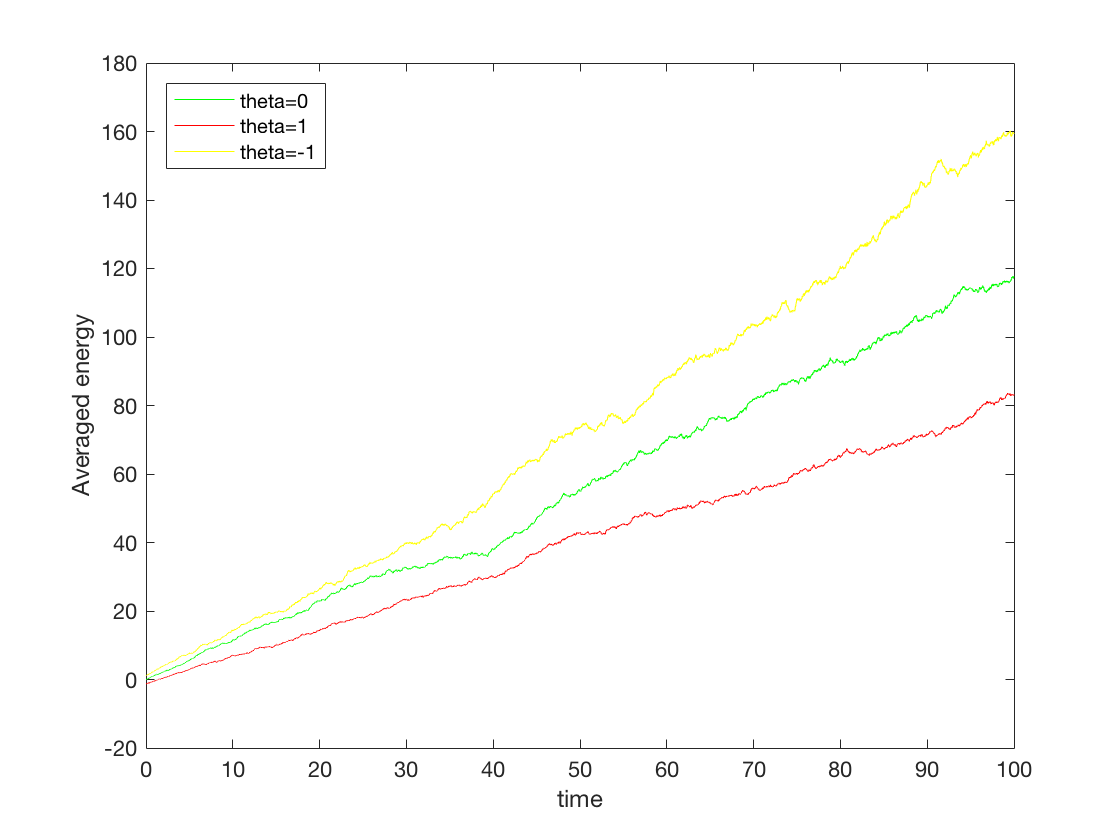}
\includegraphics[width=6cm,height=5cm]
{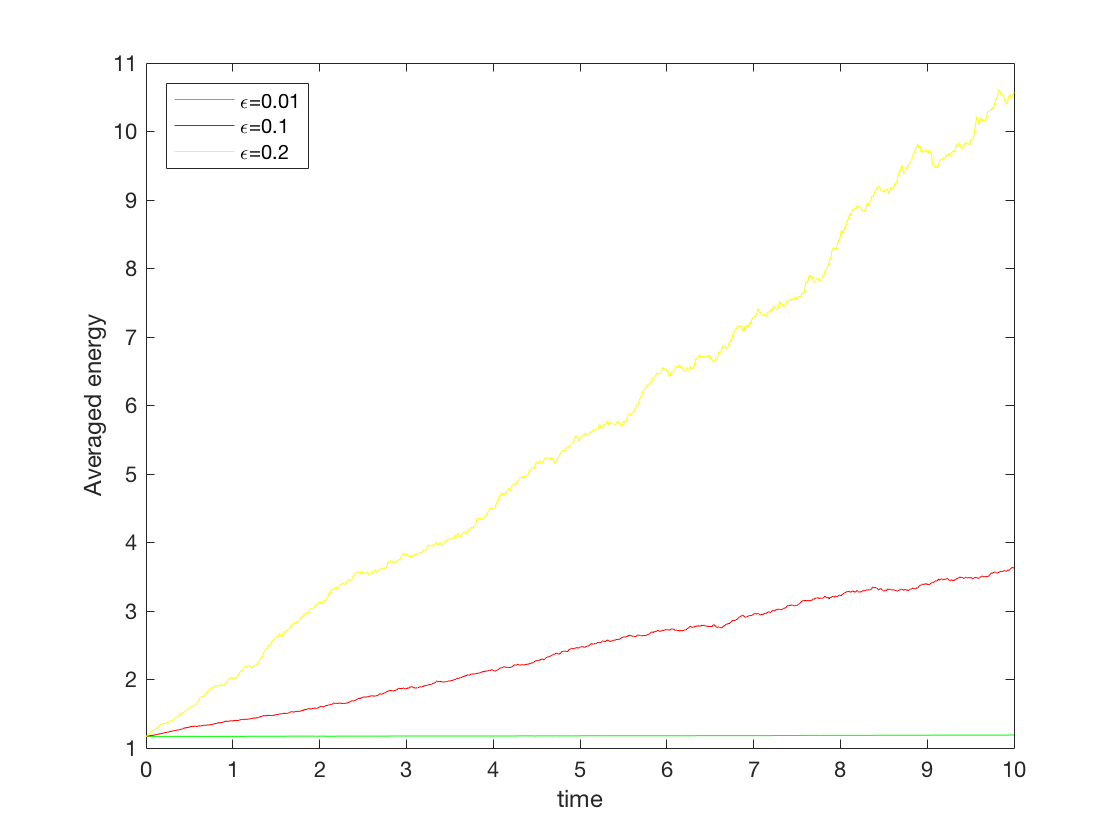}
\caption{ The evolution of the average discrete energy obeys linear growth if $\theta=-1, 0, 1, \lambda=1, \sigma=1$ (left) ;  $\epsilon=0.01, 0.1, 0.2,  \theta=-1, \lambda=1, \sigma=1$ (right).}
\end{figure}

\bibliographystyle{plain}
\bibliography{bib}

\end{document}